\documentclass[11pt]{article}

\usepackage{enumerate, amsmath, amsthm, amsfonts, amssymb, xy,  mathrsfs, graphicx, paralist, lscape, multicol}

\usepackage{tocloft}
\usepackage[usenames, dvipsnames]{color}
\usepackage[margin=1in]{geometry} 
\usepackage[colorlinks=true, urlcolor=blue, linkcolor=blue, citecolor=blue, bookmarksdepth=2]{hyperref}
\usepackage{fancyvrb}
\usepackage{multicol}
\usepackage[all]{hypcap}

\input xy
\xyoption{all}

\usepackage{mathtools}
\usepackage{tikz}
\usetikzlibrary{positioning}
\usetikzlibrary{chains}

\tikzset{node distance=2em, ch/.style={circle,draw,on chain,inner sep=2pt},chj/.style={ch,join},every path/.style={shorten >=4pt,shorten <=4pt},line width=1pt,baseline=-1ex}

\newcommand{\mlabel}[1]{ \footnotesize \(#1\) }
\newcommand{\dnode}[2][chj]{ \node[#1,label={below:\mlabel{#2}}] {}; }

\newcommand{\dnodebr}[1]{ \node[chj,label={below right:\mlabel{#1}}] {}; }

\allowdisplaybreaks

\numberwithin{equation}{section}
\newtheorem{theorem}[equation]{Theorem}

\newtheorem{proposition}[equation]{Proposition}
\newtheorem{lemma}[equation]{Lemma}
\newtheorem{corollary}[equation]{Corollary}
\newtheorem{conjecture}[equation]{Conjecture}

\theoremstyle{definition}
\newtheorem{rmk}[equation]{Remark}
\newenvironment{remark}[1][]{\begin{rmk}[#1] \pushQED{\qed}}{\popQED \end{rmk}}
\newtheorem{eg}[equation]{Example}

\newtheorem{defn}[equation]{Definition}

\newcommand{\bA}{\mathbf{A}}

\newcommand{\rA}{\mathrm{A}}

\newcommand{\cB}{\mathcal{B}}

\newcommand{\rB}{\mathrm{B}}

\newcommand{\bC}{\mathbf{C}}
\newcommand{\cC}{\mathcal{C}}

\newcommand{\rE}{\mathrm{E}}

\newcommand{\bF}{\mathbf{F}}

\newcommand{\bG}{\mathbf{G}}

\newcommand{\rG}{\mathrm{G}}

\newcommand{\cH}{\mathcal{H}}

\newcommand{\rH}{\mathrm{H}}

\newcommand{\cI}{\mathcal{I}}

\newcommand{\cK}{\mathcal{K}}

\newcommand{\cL}{\mathcal{L}}

\newcommand{\rM}{\mathrm{M}}

\newcommand{\rN}{\mathrm{N}}

\newcommand{\cO}{\mathcal{O}}

\newcommand{\bP}{\mathbf{P}}
\newcommand{\cP}{\mathcal{P}}

\newcommand{\cQ}{\mathcal{Q}}

\newcommand{\bS}{\mathbf{S}}

\newcommand{\cX}{\mathcal{X}}

\newcommand{\cY}{\mathcal{Y}}

\newcommand{\bZ}{\mathbf{Z}}

\newcommand{\bc}{\mathbf{c}}

\newcommand{\fe}{\mathfrak{e}}

\newcommand{\fg}{\mathfrak{g}}

\newcommand{\fh}{\mathfrak{h}}

\newcommand{\bv}{\mathbf{v}}

\newcommand{\bz}{\mathbf{z}}

\renewcommand{\phi}{\varphi}
\renewcommand{\emptyset}{\varnothing}

\renewcommand{\tilde}[1]{\widetilde{#1}}
\newcommand{\ol}[1]{\overline{#1}}

\newcommand{\arxiv}[1]{\href{http://arxiv.org/abs/#1}{{\tt arXiv:#1}}}

\makeatletter
\def\Ddots{\mathinner{\mkern1mu\raise\p@
\vbox{\kern7\p@\hbox{.}}\mkern2mu
\raise4\p@\hbox{.}\mkern2mu\raise7\p@\hbox{.}\mkern1mu}}
\makeatother

\DeclareMathOperator{\coker}{coker}

\DeclareMathOperator{\rank}{rank}

\DeclareMathOperator{\Pf}{Pf}

\DeclareMathOperator{\Proj}{Proj}

\DeclareMathOperator{\Jac}{Jac}
\DeclareMathOperator{\Pic}{Pic}

\newcommand{\GL}{\mathbf{GL}}
\newcommand{\SL}{\mathbf{SL}}
\newcommand{\Sp}{\mathbf{Sp}}

\newcommand{\Gr}{\mathbf{Gr}}

\newcommand{\fsl}{\mathfrak{sl}}

\newcommand{\BHM}{{\rm BHM}}
\newcommand{\CS}{\bC\bS}
\newcommand{\rCS}{\mathrm{CS}}

\title{Alternating trilinear forms on a 9-dimensional space and degenerations of $(3,3)$-polarized Abelian surfaces}
\author{Laurent Gruson \and Steven V Sam}
\date{September 15, 2014}

\begin{document}

\maketitle

\begin{abstract}
We give a detailed analysis of the semisimple elements, in the sense of Vinberg, of the third exterior power of a 9-dimensional vector space over an algebraically closed field of characteristic different from 2 and 3. To a general such element, one can naturally associate an Abelian surface $X$, which is embedded in 8-dimensional projective space. We study the  combinatorial structure of this embedding and explicitly recover the genus 2 curve whose Jacobian variety is $X$. We also classify the types of degenerations of $X$ that can occur. Taking the union over all Abelian surfaces in Heisenberg normal form, we get a $5$-dimensional variety which is a birational model for a genus $2$ analogue of Shioda's modular surfaces. We find determinantal set-theoretic equations for this variety and present some additional equations which conjecturally generate the radical ideal.
\end{abstract}


\setcounter{tocdepth}{1}
\tableofcontents

\section*{Introduction}

For the purposes of this introduction, we work over the complex numbers $\bC$. In the notation of Shephard and Todd \cite{shephard}, the group $\rG_{32} \cong \bZ/3 \times \Sp_4(\bF_3)$ acts on its $4$-dimensional reflection representation, and it is known from the work of several mathematicians \cite{burkhardt, coble1, geer} that the regular orbits of this action parametrize Abelian surfaces embedded in $\bP^8$ via an indecomposable $(3,3)$-polarization together with a marked odd theta characteristic. Here regular means that the orbit does not intersect any of the reflection hyperplanes of $\rG_{32}$. In fact, this action can be seen as a slice in a bigger action: that of $\SL_9(\bC)$ on $\bigwedge^3 \bC^9$. This is just one of the examples of a $\theta$-representation in the sense of Vinberg \cite{vinberg}. In joint work with Weyman \cite{gsw}, the authors described, for many ``sporadic'' examples (mostly those that had not previously appeared in the literature) a construction of Abelian varieties from orbits in a $\theta$-representation. In the case of $\SL_9(\bC)$ acting on $\bigwedge^3 \bC^9$, one produces the Abelian surfaces in $\bP^8$ mentioned before.

The advantage of this bigger representation is that it contains richer structure. For example, every element in $\bigwedge^3 \bC^9$ has a decomposition as a sum of a ``semisimple'' element and a ``nilpotent'' element, and the elements in the slice correspond precisely to the semisimple elements, i.e., those whose nilpotent part is $0$. In particular, one can extend the classical parameterization above by considering the construction of \cite{gsw} for these other elements and see which geometric objects arise. Another point is that if we think of belonging to a slice as being ``diagonalizable'' (i.e., having maximal possible split rank), then the notions of diagonalizable and semisimple do not agree over fields which are not algebraically closed. The construction in \cite{gsw} is valid over any field, and for arithmetic applications, it is relevant to understand these constructions over fields such as the rational numbers. 

But as a first step, it is important to understand the situation over an algebraically closed field and with a semisimple element. Using the slice description and the construction in \cite{gsw}, one can associate a geometric object to any semisimple orbit, not just the regular ones. In this article, we apply this to the orbits of $\rG_{32}$ on its reflection representation $\fh$ and analyze the geometric objects constructed from non-regular orbits. These are certain degenerations of the Abelian surfaces and we give explicit descriptions of the kinds of degenerations that appear. The family of Abelian surfaces that we construct naturally form a flat family over an open subset of $\fh$ which one might consider analogous to Shioda's modular surfaces. We study the projection of this family into $\bP^8$ and describe equations for the closure. We would also like to point out that one advantage of the slice is that it allows a study of the moduli spaces considered in this paper from the perspective of tropical geometry. See \cite{RSS} for developments in this direction.

We now review the contents of this paper. In \S\ref{sec:background} we summarize some information about Vinberg's $\theta$-representations as well as the role of Heisenberg groups in finding normal forms for equations of embedded Abelian varieties. In \S\ref{sec:warmup} we review an easier situation that is analogous to our situation: quintic elliptic normal curves. Most of these results are known through other work, but we provide different proofs that will serve as simpler versions of some of the proofs that will appear in \S\ref{sec:shioda}. Here the group $\rG_{32}$ is replaced by the group $\rG_{16} \cong \bZ/5 \times \SL_2(\bF_5)$ acting on its $2$-dimensional reflection representation. 

In \S\ref{sec:33} we review the construction of \cite{gsw} in our case and also collect some background material on $(3,3)$-polarized Abelian surfaces. The important main result from existing literature is Theorem~\ref{thm:M23minus} which is a precise version of the parametrization of Abelian surfaces mentioned earlier. We also present some results which we have not seen in the literature. We wish to highlight Theorem~\ref{thm:fanoP4} which identifies our Abelian surfaces with Fano varieties of $7$-dimensional cubic hypersurfaces (the Coble cubic of the Abelian surface), and Theorem~\ref{thm:summary} which offers a summary of how this result connects Theorem~\ref{thm:M23minus} to an explicit universal family of genus $2$ curves. 

In \S\ref{sec:w39degenerations} we analyze the degenerations of these Abelian surfaces mentioned earlier. A summary is given in the beginning of \S\ref{sec:w39degenerations}. Finally, in \S\ref{sec:shioda}, we study an analogue of Shioda's modular surfaces. We take a naive point of view on these surfaces here: in low degrees, one can describe certain birational models of Shioda's modular surfaces as the union of all elliptic normal curves whose equations are in ``Heisenberg normal form''. So here we study the 5-dimensional variety which is the union of the (degenerate) Abelian surfaces in Heisenberg normal form. We describe determinantal set-theoretic equations for this variety and offer a conjecture on its full prime ideal.

\subsection*{Notation and conventions}

\begin{compactitem}
\item Throughout we will work with projective varieties which are usually embedded in a projective space, or a product of projective spaces. When we refer to properties, such as Cohen--Macaulay or Gorenstein, we mean that these properties hold for the local rings. When we wish to refer to the singularity of the vertex of the affine cone over a variety, we will use the prefix ``arithmetically'', e.g., arithmetically Cohen--Macaulay.

\item Given a vector space $V$, we use $\Gr(k,V)$ to denote the Grassmannian of $k$-dimensional subspaces of $V$. We set $\bP(V) = \Gr(1,V)$, so that our projective spaces parametrize lines rather than hyperplanes.

\item Our notation for graded Betti tables follows {\tt Macaulay 2} \cite{m2} notation.

\item A diagonalizable linear automorphism is called a complex reflection if all but one of its eigenvalues are $1$, and the remaining eigenvalue is a root of unity. A group generated by complex reflections is a complex reflection group. We will not restrict ourselves to the complex numbers (or even fields of characteristic $0$); the choice of terminology is for historical purposes. To index the finite complex reflection groups $\rG_i$, we use the notation of Shephard--Todd \cite{shephard}.

\item The notations $\SL_n(K)$, $\GL_n(K)$, $\Sp_n(K)$ refer to special linear, general linear, and symplectic groups, respectively, consisting of matrices of size $n$ (so for $\Sp_n$, we require that $n$ is even) over a field $K$. For our purposes, $K$ is either an algebraically closed field or a finite field, so there will be no ambiguity about which forms of the group to take.
\end{compactitem}

\subsection*{Acknowledgements} 

Steven Sam was supported by an NDSEG fellowship and a Miller research fellowship while this work was done.
The software packages {\tt GAP} \cite{gap} and {\tt Macaulay 2} \cite{m2} were very helpful.
The authors thank Igor Dolgachev and Bernd Sturmfels for helpful discussions.

\section{Background} \label{sec:background}

\subsection{Vinberg $\theta$-representations}

As shown by Kostant \cite{kostant}, the action of a semisimple complex Lie group $G$ on its adjoint representation $\fg$ has many favorable properties: its ring of invariants is a polynomial algebra, and is isomorphic to the ring of invariants of a finite reflection group (Weyl group) on a smaller vector space (a Cartan subalgebra). This situation was generalized by Kostant and Rallis \cite{kostantrallis} as follows: given an involution $\theta$ on the group $G$, the fixed point subgroup $G^\theta$ acts on the eigenspaces $\fg_{\pm}$ under the induced action of $\theta$ on $\fg$, and similar properties hold for the action of $G^\theta$ on $\fg_{-}$ (the action on $\fg_+$ reduces to the previous case).

In \cite{vinberg}, Vinberg investigated the situation of an arbitrary finite order automorphism $\theta$ on $G$ and showed that the properties mentioned above still hold: the ring of invariants of $G^\theta$ on each eigenspace of $\theta$ on $\fg$ is a polynomial ring and is equivalent to the ring of invariants of a finite group (we will call this the Weyl group for $(G,\theta)$) acting on a smaller vector space, which we will call a Cartan subspace. Furthermore, each Cartan subspace is the intersection of the eigenspace with a Cartan subalgebra in $\fg$, and all of them are conjugate under the action of $G^\theta$. Since it will not affect what we do, we will replace $G^\theta$ by its simply-connected cover.

We will be interested in one particular example of an order 3 automorphism on the adjoint complex Lie group of type $\rE_8$. In this case, the simply-connected cover of $G^\theta$ is $\SL_9(\bC)$ and the nontrivial eigenspaces are isomorphic to $\bigwedge^3 \bC^9$ and $\bigwedge^6 \bC^9$, respectively. The finite group in this case is the Shephard--Todd complex reflection group denoted $\rG_{32}$ \cite{shephard}. Actually, we can replace $\bC$ with any algebraically closed field $K$ of characteristic different from $2$ and $3$.

We study the orbits of $\SL_9(K)$ on $\bigwedge^3(K^9)$ and given a sufficiently generic orbit, we will construct an Abelian surface embedded in $\bP^8$. This construction is explained in \cite[\S 5]{gsw}. To get the most information, we restrict to the case of semisimple orbits (those orbits whose elements belong to some Cartan subspace) and work with a specific Cartan subspace (since they are all conjugate, we have not lost anything by doing this). In this situation, an element is considered sufficiently generic if it lies outside of the union of the reflection hyperplanes of $\rG_{32}$. We will describe the degenerations of the Abelian surfaces when we move into the hyperplanes. The possible types of degenerations correspond to the various types of intersections of hyperplanes.

For arithmetic applications, it is natural to investigate the structure of semisimple orbits over fields which are not algebraically closed, and their associated Abelian surfaces. In particular, it will no longer be true that two Cartan subspaces are conjugate to one another via $G^\theta(K)$ when $K$ is not algebraically closed, and one has to understand finer invariants such as split rank. It is also natural to ask about orbits which are not semisimple. We leave these issues for future work.

As a warmup, in \S\ref{sec:warmup}, we will also consider an order $5$ automorphism of the adjoint complex Lie group of type $\rE_8$, in which case we are considering the action of the group $\SL_5(K) \times \SL_5(K)$ on $K^5 \otimes \bigwedge^2 K^5$. In this case, one can naturally associate an elliptic quintic curve to each generic orbit.

\subsection{Equations of Abelian varieties} \label{sec:heisenberg}

We will need some information about embeddings of Abelian varieties $X$ into projective space via an ample line bundle $\cL$, which can be found in \cite[Chapter 6]{birkenhake} (this reference only deals with the complex numbers; for fields of positive characteristic, one can consult \cite{mumford}). Given $x \in X$, let $t_x$ be translation by $x$. Then the subgroup $K(\cL) = \{x \in X \mid t_x^* \cL \cong \cL\}$ is a finite group of the form $(\bZ/d_1 \oplus \cdots \oplus \bZ/d_g)^2$ where $g = \dim X$. 

For our elliptic normal quintic, we have $K(\cL) = (\bZ/5)^2$ and for our Abelian surfaces in $\bP^8$, we will have $K(\cL) = (\bZ/3)^4$ \cite[Proposition 5.6]{gsw}. With respect to this group, there are coordinates on the ambient projective space, called Schr\"odinger coordinates, on which the group $K(\cL)$ acts in a simple way. We just explain it in our two cases.

For the case $K(\cL) = (\bZ/5)^2$, there are coordinates $z_1, \dots, z_5$ (the indices are taken modulo $5$) for which the generators of $K(\cL)$ are given by the operators $\sigma \colon z_i \mapsto z_{i+1}$ and $\tau \colon z_i \mapsto \zeta^i z_i$ (here $\zeta$ is a fixed primitive $5$th root of unity). This is a projective representation of $(\bZ/5)^2$ on $\bP^4$, and we can lift it to a linear representation of a central extension
\[
1 \to \bZ/5 \to H_5 \to (\bZ/5)^2 \to 1
\]
where the central $\bZ/5$ is generated by scalar multiplication by $5$th roots of unity. If $N(H_5)$ is the normalizer of $H_5$ in $\GL_5(K)$, then we have $N(H_5) / H_5 \cong \SL_2(\bF_5) = \Sp_2(\bF_5)$. 

There is a natural symplectic pairing on $K(\cL) = (\bZ/5)^2$: for two operators $\phi, \psi \in K(\cL)$, their commutator $\phi \psi \phi^{-1} \psi^{-1}$ is a scalar operator which is multiplication by a $5$th root of unity. In particular, it is a power of $\zeta$, and we take the value of the pairing to be this power (which is defined modulo $5$). We define a {\bf full level 5 structure} on $X$ to be a choice of symplectic basis for $K(\cL)$, i.e., $b_1,b_2$ such that $\langle b_1, b_2 \rangle = 1 \pmod 5$, with the extra caveat that we identify the basis $b_1, b_2$ with $-b_1, -b_2$. So the finite group $\bP\SL_2(\bF_5) = \SL_2(\bF_5) / \{\pm I_2 \}$ acts simply transitively on the set of level structures. The ordered basis $\tau, \sigma$ is a symplectic basis.

For the case $K(\cL) = (\bZ/3)^4$, there are coordinates $x_{i,j}$ ($i,j \in \bZ/3$) for which the generators of $K(\cL)$ are given by 
\begin{align*}
\sigma_1 \colon x_{i,j} \mapsto x_{i+1,j}, \qquad \sigma_2 \colon x_{i,j} \mapsto x_{i,j+1}, \\
\tau_1 \colon x_{i,j} \mapsto \omega^i x_{i,j}, \qquad \tau_2 \colon x_{i,j} \mapsto \omega^j x_{i,j}.
\end{align*}
(here $\omega$ is a fixed primitive cube root of unity). This is a projective representation of $(\bZ/3)^4$ on $\bP^8$ which lifts to a linear representation of a central extension
\[
1 \to \bZ/3 \to H_{3,2} \to (\bZ/3)^4 \to 1
\]
where we have added scalar multiplication by cube roots of unity. In this case, we opt to use $z_1, \dots, z_9$ instead of $x_{0,0}, \dots, x_{2,2}$, and a dictionary between the two notations is given in \eqref{eqn:affineplane}. If $N(H_{3,2})$ is the normalizer of $H_{3,2}$ in $\GL_9(K)$, then we have $N(H_{3,2}) / H_{3,2} \cong \Sp_4(\bF_3)$.

Again, we have a natural symplectic pairing on $K(\cL)$ and we can define {\bf full level 3 structures} for $X$ as before. The finite group $\bP\Sp_4(\bF_3) = \Sp_4(\bF_3) / \{\pm I_4 \}$ acts simply transitively on the set of level structures.

The groups $H_5$ and $H_{3,2}$ are two instances of {\bf Heisenberg groups}. The advantage of this coordinate system is that it allows us to work with normal forms of the equations of the Abelian variety, which will be explained in more detail in the corresponding sections. 

\section{Warmup: Elliptic quintic curves} \label{sec:warmup}

Let $K$ be an algebraically closed field. Some of the statements in this section require that the characteristic is different from $5$, specifically, we will use the existence of primitive $5$th roots of unity in some places.

Much of what we say in this section is already contained in \cite{bhm}, but we offer some different proofs and additional calculations. We also refer the reader to \cite{fisher} for algorithmic aspects on invariants of elliptic quintic curves.

\subsection{The Vinberg representation}

Let $A$ and $B$ be two 5-dimensional vector spaces with bases $a_1, \dots, a_5$ and $b_1, \dots, b_5$. We set $G = \SL(A) \times \SL(B)$. The $G$-module $A \otimes \bigwedge^2 B$ is an example of a Vinberg representation: there is an order 5 automorphism of the Lie algebra of type $\rE_8$, which corresponds to the simple root $\alpha_5$ of the affine Dynkin diagram in Bourbaki notation
\begin{align} \label{eqn:affineE8}
\begin{tikzpicture}
\begin{scope}[start chain]
\dnode{1}
\dnode{3}
\dnode{4}
\dnode{5}
\dnode{6}
\dnode{7}
\dnode{8}
\dnode{0}
\end{scope}
\begin{scope}[start chain=br going above]
\chainin(chain-3);
\dnodebr{2}
\end{scope}
\end{tikzpicture}
\end{align}
(see \cite[\S 8.6]{kac} for details on finite order automorphisms of simple Lie algebras), and whose eigenspace decomposition is
\[
(\fsl(A) \times \fsl(B)) \oplus (A \otimes \bigwedge^2 B) \oplus (\bigwedge^2 A \otimes \bigwedge^4 B) \oplus (\bigwedge^3 A \otimes B) \oplus (\bigwedge^4 A \otimes \bigwedge^3 B)
\]
The maximal dimension of a Cartan subspace in $A \otimes \bigwedge^2 B$ is $2$ and the Weyl group is Shephard--Todd group $\rG_{16}$ \cite[\S 9]{vinberg}. We note that $\rG_{16} \cong \bZ/5 \times \SL_2(\bF_5)$.

For our choice of Cartan subspace for $A \otimes \bigwedge^2 B$, we pick the subspace with basis
\begin{align*}
h_1 &= \phantom{-}a_1 \otimes b_3 \wedge b_4 + a_2 \otimes b_4 \wedge b_5 - a_3 \otimes b_1 \wedge b_5 + a_4 \otimes b_1 \wedge b_2 + a_5 \otimes b_2 \wedge b_3, \\
h_2 &= -a_1 \otimes b_2 \wedge b_5 + a_2 \otimes b_1 \wedge b_3 + a_3 \otimes b_2 \wedge b_4 + a_4 \otimes b_3 \wedge b_5 - a_5 \otimes b_1 \wedge b_4,
\end{align*}
so a general semisimple element is conjugate to one of the form 
\begin{align} \label{eqn:5w25general}
c_1 h_1 + c_2 h_2,
\end{align}
for some $c_1,c_2 \in K$, and two such elements are conjugate under $G$ if and only if they are conjugate by the action of the Weyl group $W = \rG_{16}$. We list a pair of generators for the $W$-action:
\[
\mu = \begin{pmatrix} 0 & -1 \\ 1 & 0 \end{pmatrix}, \qquad 
\nu = \frac{1}{5} \begin{pmatrix}  4\zeta+3\zeta^2+2\zeta^3+\zeta^4 & -2\zeta-4\zeta^2-\zeta^3-3\zeta^4\\
-2\zeta+\zeta^2-\zeta^3-3\zeta^4 & \zeta+2\zeta^2+3\zeta^3-\zeta^4\end{pmatrix}
\]
($\zeta$ is a fixed primitive $5$th root of unity).

This Cartan subspace has the following description. We have an action of the Heisenberg group $H_5$ (see \S\ref{sec:heisenberg}) on $A$, and we also let it act on $B$ in the same way (i.e., replacing $a_i$ with $b_i$). Then the Cartan subspace is the invariant subspace under $H_5$ in $A \otimes \bigwedge^2 B$. From this, we also get the action of $N(H_5)/H_5 \cong \SL_2(\bF_5)$ on the Cartan subspace. If we also add the scalar matrices corresponding to the 5th roots of unity, then we get the action of the group $W = \rG_{16}$.

\subsection{Degree 5 curves of genus 1} \label{ss:deg5genus1}

We represent the element \eqref{eqn:5w25general} by a $5 \times 5$ skew-symmetric matrix over $\bP(A^*)$, where the coordinate functions are $z_i = a_i$
\begin{align} \label{eqn:ellipticmat}
\Psi_{\bc}(z) = \begin{pmatrix} 0 & c_1z_4 & c_2z_2 & -c_2z_5 & -c_1z_3\\ -c_1z_4 & 0 & c_1z_5 & c_2z_3 & -c_2z_1\\ -c_2z_2 & -c_1z_5 & 0 & c_1z_1 & c_2z_4\\ c_2z_5 & -c_2z_3 & -c_1z_1 & 0 & c_1z_2\\ c_1z_3 & c_2z_1 & -c_2z_4 & -c_1z_2 & 0
\end{pmatrix}.
\end{align}
The $4 \times 4$ Pfaffians are
\begin{align} \label{eqn:deg5pfaff}
\begin{array}{lll}
c_1c_2 z_1^2 - c_1^2 z_2z_5 + c_2^2 z_3z_4, & 
c_1c_2 z_2^2 - c_1^2 z_1z_3 + c_2^2 z_4z_5, & 
c_1c_2 z_3^2 - c_1^2 z_2z_4 + c_2^2 z_1z_5,\\
c_1c_2 z_4^2 - c_1^2 z_3z_5 + c_2^2 z_1z_2, &
c_1c_2 z_5^2 - c_1^2 z_1z_4 + c_2^2 z_2z_3,
\end{array}
\end{align}
and for $(c_1,c_2) \ne (0,0)$, these equations define a curve of arithmetic genus $1$.

Let $\zeta$ be a primitive $5$th root of unity. The curve above is nonsingular unless $c_1/c_2$ is one of the $12$ values $0, \infty, \zeta^i(1 \pm \sqrt{5})/2$ for $i=0,\dots,4$ \cite[\S 2.5]{bhm}. This is the union of the reflection lines for $\rG_{16}$ in the representation spanned by $c_1, c_2$. Letting $\Delta$ be the product of the corresponding linear forms, we rewrite this condition as $\Delta(\bc) \ne 0$.

Using the Buchsbaum--Eisenbud classification of codimension 3 Gorenstein ideals, it follows that every normally embedded degree 5 genus 1 curve in $\bP^4$ can be obtained from this construction (see \cite[Example 1.10, \S 4.2]{gsw} for details and references). In conclusion, if $\fh^\circ$ is the complement of the $12$ reflection lines in the Cartan subspace $\fh$, then we see that $\bP(\fh^\circ)$ parametrizes smooth genus $1$ degree $5$ curves in $\bP^4$ with full level $5$ structure. The scalar matrices in $\rG_{16}$ are generated by $\bZ/5$ and $\{\pm I_2 \} \subset \SL_2(\bF_5)$, so the free action of $\rG_{16}$ on $\fh^\circ$ descends to a free action of $\bP\SL_2(\bF_5)$ on $\bP(\fh^\circ)$, and taking the quotient $\bP(\fh^\circ) / \bP\SL_2(\bF_5)$ has the effect of forgetting the level structure. Here we just need to note that the kernel of the map $\rG_{16} \to \bP\GL_3(K)$ induced by the action of $\rG_{16}$ on the functions $[c_1c_2:c_1^2:c_2^2]$ is generated by $\bZ/5$ and $\{\pm I_2\}$.

\begin{remark}
Since the ring of invariants of $A \otimes \bigwedge^2 B$ under $G$ is the same as the ring of invariants of $\fh$ under $W$, we get a similar interpretation for orbits in an open subset of $\bP(A \otimes \bigwedge^2 B)$. It would be interesting to give an interpretation (similar to level structure) for the actual vectors of the open subset of $\bP(A \otimes \bigwedge^2 B)$.
\end{remark}

\begin{remark}
It is also interesting to study the Pfaffian loci for nilpotent vectors (i.e., not belonging to the Cartan). The generic nilpotent orbit has the representative
\[
[5;15]+[2;34]+[1;12]+[3;35]+[5;24]+[4;13]+[2;25]+[4;45],
\]
where $[i;jk]$ means $a_i \otimes b_j \wedge b_k$. It gives the matrix
\[
\begin{pmatrix}0& z_{1}& z_{4}& 0&  z_{5}\\
  {-z_{1}}& 0& 0& z_{5}&  z_{2}\\
  {-z_{4}}& 0& 0& z_{2}&  z_{3}\\
  0& {-z_{5}}& {-z_{2}}& 0&  z_{4}\\
  {-z_{5}}& {-z_{2}}& {-z_{3}}& {-z_{4}}& 0
\end{pmatrix}.
\]
The $4 \times 4$ Pfaffians give the ideal $(z_{1} z_{2}-z_{4}z_{5}, z_{1}z_{3}-z_{2} z_{4}, z_{1} z_{4}+z_{5}^{2}, z_{4}^{2}+z_{2} z_{5},z_{2}^{2}-z_{3} z_{5})$ which defines a rational cuspidal quintic whose cusp point
is $[0:0:1:0:0]$. It has the parametrization
\[
[s:t] \mapsto [\sqrt{-1} t^5 : s^3t^2 : s^5 : \sqrt{-1} s^2t^3 : st^4] .
\]
The group of automorphisms of this curve that extend to automorphisms of $\bP^4$ is generated by scaling $t$, so the orbit of this curve in $A \otimes \bigwedge^2 B$ has codimension $2$ (since the group $G$ we use has dimension $49$), which verifies that its orbit closure is a component of the nullcone (in fact, the nullcone is irreducible in this case).

Some other degenerations of the elliptic quintic curve are explored in \cite[\S 4.3]{gsw}.
\end{remark}

\subsection{Shioda's level 5 modular surface}

If we treat $c_1$ and $c_2$ as new coordinates, the $5$ equations \eqref{eqn:deg5pfaff} define Shioda's modular surface $S(5) \subset \bP^1 \times \bP(A^*)$. More precisely, define
\begin{align*}
S(5) &= \{(\bc, z) \in \bP^1 \times \bP(A^*) \mid \rank(\Psi_{\bc}(z)) < 4\},\\
S(5)^\circ &= \{(\bc, z) \in \bP^1 \times \bP(A^*) \mid \rank(\Psi_{\bc}(z)) < 4,\ \Delta(\bc) \ne 0\}.
\end{align*}
The projection of $S(5)$ to $\bP(A^*)$ is a birational model for $S(5)$ which we denote by $S(5)_{15}$ (since it has degree $15$). 

\begin{proposition}
$S(5)$ and $S(5)_{15}$ are irreducible surfaces.
\end{proposition}

\begin{proof}
As discussed in \S\ref{ss:deg5genus1}, the fibers of $S(5)^\circ \to \bP^1$ are smooth irreducible curves, and hence $S(5)^\circ$ is irreducible \cite[Exercise 14.3]{eisenbud} of codimension $3$. Since $S(5)$ is defined as a Pfaffian degeneracy locus, each of its irreducible components has codimension at most 3 by the generic perfection theorem \cite[Theorem 3.5]{bv}. The fibers of $S(5) \setminus S(5)^\circ$ over $\bP^1$ are singular curves, so $S(5) \setminus S(5)^\circ$ is $1$-dimensional. In particular, it cannot contribute an irreducible component, and so $S(5)$ is irreducible. The same is true for $S(5)_{15}$ since it is the image of $S(5)$.
\end{proof}

\begin{theorem}[{Barth--Hulek--Moore \cite[\S 1.1]{bhm}}] \label{thm:shiodasurface}
The prime ideal of $S(5)_{15}$ is generated by the maximal minors of the matrix
\[
\BHM(z) = \begin{pmatrix}
z_1^2 & z_2^2 & z_3^2 & z_4^2 & z_5^2 \\ 
z_2z_5 & z_1z_3 & z_2z_4 & z_3z_5 & z_1z_4\\
z_3z_4 & z_4z_5 & z_1z_5 & z_1z_2 & z_2z_3
\end{pmatrix}.
\]
The singular locus of $S(5)_{15}$ is set-theoretically defined by the $2 \times 2$ minors of $\BHM(z)$. (For the full ideal, see Remark~\ref{rmk:S5-comments}.)
\end{theorem}

\begin{proof}
If $[z_1 : \cdots : z_5]$ lies on an elliptic normal quintic for some $c_1, c_2$, then $\BHM(z)$ does not have full rank (the linear dependence among the rows being given by $(c_1 c_2, -c_1^2, c_2^2)$), so $S(5)_{15}$ is contained in the vanishing locus of the maximal minors.

Conversely, suppose that for some $[z_1 : \cdots : z_5]$, $\BHM(z)$ does not have full rank, but suppose that it has rank $2$. For each choice of two columns $i,j$, we can construct linear dependencies $v(i,j)$ among the rows using Laplace expansion and the $2 \times 2$ minors for those two columns. For example, here is the linear dependence if we choose the first two columns:
\[
v(1,2) = (z_1z_3^2z_4 - z_2z_4z_5^2, z_1^2 z_4z_5 - z_2^2 z_3z_4, z_2^3z_5 - z_1^3z_3)
\]
One can check directly (say with {\tt Macaulay 2}) that for all $i,j$, we have $v(i,j)_1^2 = - v(i,j)_2v(i,j)_3$ modulo the ideal of maximal minors of $\BHM(z)$. But this says that $v(i,j) = (c_1 c_2, -c_1^2, c_2^2)$ for some $c_1,c_2$. Since $\BHM(z)$ has rank $2$, we can find $i,j$ so that the linear dependence is not identically $0$, which gives us $(c_1, c_2) \ne (0,0)$.

To check that the ideal is radical, the example is small enough that it can be done directly with the {\tt radical} command in {\tt Macaulay 2}. 

The statements about the locus where $\BHM(z)$ has rank $1$ follow by direct calculation.
\end{proof}

\begin{remark} \label{rmk:S5-comments}
Here is the explicit computation of the ideal in {\tt Macaulay 2} ($I$ denotes the ideal generated by the $5$ equations \eqref{eqn:deg5pfaff}):
\begin{Verbatim}[samepage=true]
eliminate({c_1,c_2}, quotient(I, ideal(c_1*c_2)))
\end{Verbatim}

\begin{enumerate}
\item The graded Betti table of $S(5)_{15}$ is a ``fake Gorenstein'' complex:
\small \begin{Verbatim}[samepage=true]
       0  1  2 3
total: 1 10 10 1
    0: 1  .  . .
       ...
    5: . 10 10 .
    6: .  .  . .
    7: .  .  . 1
\end{Verbatim}
\normalsize 
Its degree is $15$ and its Hilbert series is
\[
\frac{1+2 T+3 T^{2}+4 T^{3}+5 T^{4}+6 T^{5}-3 T^{6}-2 T^{7}-T^{8}}{({1-T})^{3}}.
\]
Here is the graded Betti table of $\coker \BHM(z)$, which is Cohen--Macaulay:
\small \begin{Verbatim}[samepage=true]
       0 1 2
total: 3 5 2
    0: 3 . .
    1: . 5 .
    2: . . .
    3: . . 2
\end{Verbatim}
\normalsize 
The second differential in this complex is
\[
\begin{pmatrix}
-z_{2}^{2} z_{4}+z_{3} z_{5}^{2}& z_{2} z_{3}^{2}-z_{4}^{2} z_{5}\\
      z_{1}^{2} z_{4}-z_{3}^{2} z_{5}&      z_{3} z_{4}^{2}-z_{1} z_{5}^{2}\\
      -z_{1} z_{4}^{2}+z_{2}^{2} z_{5}&      -z_{1}^{2} z_{2}+z_{4} z_{5}^{2}\\
      z_{1} z_{3}^{2}-z_{2} z_{5}^{2}&      -z_{2}^{2} z_{3}+z_{1}^{2} z_{5}\\
      -z_{1}^{2} z_{3}+z_{2} z_{4}^{2}&      z_{1} z_{2}^{2}-z_{3}^{2} z_{4}
      \end{pmatrix}.
\]
Its $2 \times 2$ minors generate the prime ideal of $S(5)_{15}$ and its entries generate the radical ideal that defines the $30$ singular points. 

\item The Betti table for the normalization of $S(5)_{15}$ as a module over the homogeneous coordinate ring of $\bP^4$ is
\small \begin{Verbatim}[samepage=true]
       0  1 2
total: 6 10 4
    0: 1  . .
    1: .  . .
    2: 5 10 .
    3: .  . 4
\end{Verbatim}
\normalsize
Its Hilbert series is $(1+2T+8T^2+4T^3) / (1-T)^3$. 

\item The Betti table of the radical of the ideal of $2 \times 2$ minors of $\BHM(z)$ is
\small \begin{Verbatim}[samepage=true]
       0  1  2  3 4
total: 1 10 14 10 5
    0: 1  .  .  . .
    1: .  .  .  . .
    2: . 10 10  . .
    3: .  .  4  . .
    4: .  .  . 10 5
\end{Verbatim}
\normalsize

The determinantal expression for $S(5)_{15}$ shows that every point of $S(5)_{15}$ except these $30$ points lies on a unique elliptic quintic. \qedhere
\end{enumerate}
\end{remark}

\begin{remark}
Let $\zeta$ be a primitive $5$th root of unity. The space of Heisenberg invariant quintics is $6$-dimensional and its zero locus is the union of $25$ disjoint lines: for each $i,j \in \{1,\dots,5\}$, we get a line whose equations are 
\[
z_i = z_{i+2} + \zeta^j z_{i+3} = \zeta^{2j} z_{i+1} + z_{i+4} = 0. 
\]
Using the matrix $\BHM(z)$ in Theorem~\ref{thm:shiodasurface}, we see that this line is in $S(5)_{15}$. In fact, it does not contain any of the $30$ singular points, so it intersects each genus $1$ curve (including the $12$ degenerate ones) in exactly one point. Hence the map $S(5)_{15} \dashrightarrow \bP^1$ (where the coordinates on $\bP^1$ are $[c_1:c_2]$) has a section corresponding to each of these lines. For example, if we have a point $z$ on the line
\[
z_1 = z_3 + \zeta z_4 = \zeta^2 z_2 + z_5 = 0,
\]
then the kernel of $\BHM(z)$ is $[-\zeta z_2 z_3 : -\zeta^2 z_3^2 : z_2^2]$ and this is the image of the point $[c_1:c_2] = [-\zeta z_3: z_2]$. Hence the section $\bP^1 \to S(5)_{15}$ is 
\[
[c_1:c_2] \mapsto [0:c_2:-\zeta^4 c_1: \zeta^3 c_1: -\zeta^2 c_2]. \qedhere
\]
\end{remark}

\section{$(3,3)$-polarized Abelian surfaces} \label{sec:33}

Let $K$ be an algebraically closed field of characteristic different from $2$ and $3$.

\subsection{The Vinberg representation} \label{ss:coblecubics}

Let $A$ be a $9$-dimensional vector space with basis $a_1, \dots, a_9$ and set $G = \SL(A)$. Let $z_1, \dots, z_9$ be the dual basis, which act as coordinates on the projective space $\bP(A^*)$ of lines in $A^*$ (or hyperplanes in $A$). The $G$-representation $\bigwedge^3 A$ is an example of a Vinberg representation: there is an order $3$ automorphism of the Lie algebra $\fe_8$ of type $\rE_8$, which corresponds to the simple root $\alpha_2$ of the affine Dynkin diagram in Bourbaki notation (see \eqref{eqn:affineE8} for the diagram and see \cite[\S 8.6]{kac} for details on finite order automorphisms of simple Lie algebras), and whose eigenspace decomposition is
\[
\fsl(A) \oplus \bigwedge^3 A \oplus \bigwedge^6 A.
\]
The maximal dimension of a Cartan subspace in $\bigwedge^3 A$ is $4$ and the Weyl group is Shephard--Todd group $\rG_{32}$ \cite[\S 9]{vinberg}. We note that $\rG_{32} \cong \bZ/3 \times \Sp_4(\bF_3)$. There are $80$ complex reflections in $\rG_{32}$. They are all of order $3$, so they come in pairs and define $40$ reflection hyperplanes.

As in \`Elashvili--Vinberg \cite{ev}, we can pick a Cartan subspace $\fh$ for $\bigwedge^3 A$ which has the basis
\begin{align*}
h_1 &= a_1 \wedge a_2 \wedge a_3 + a_4 \wedge a_5 \wedge a_6 + a_7 \wedge a_8 \wedge a_9,\\*
h_2 &= a_1 \wedge a_4 \wedge a_7 + a_2 \wedge a_5 \wedge a_8 + a_3 \wedge a_6 \wedge a_9,\\*
h_3 &= a_1 \wedge a_5 \wedge a_9 + a_2 \wedge a_6 \wedge a_7 + a_3 \wedge a_4 \wedge a_8,\\*
h_4 &= a_1 \wedge a_6 \wedge a_8 + a_2 \wedge a_4 \wedge a_9 + a_3 \wedge a_5 \wedge a_7.
\end{align*}
These basis vectors correspond to the points in $\bP^1_{\bF_3}$. More precisely, we fix a line, wedge together its vectors, and sum over the three translates of the lines. Here we choose to use $1, \dots, 9$ rather than $(0,0), \dots, (2,2)$ for compactness of notation, but here is a dictionary:
\begin{align} \label{eqn:affineplane}
\begin{matrix} 1 & 2 & 3 \\ 4 & 5 & 6 \\ 7 & 8 & 9 \end{matrix} =
\begin{matrix} (0,0) & (0,1) & (0,2) \\ (1,0) & (1,1) & (1,2) \\ (2,0) & (2,1) & (2,2) \end{matrix}
\end{align}
Here we have chosen an origin  in $\bA^2_{\bF_3}$, so we also have the involution 
\begin{equation} \label{eqn:iota}
\iota \colon x_{i,j} \mapsto x_{-i,-j},
\end{equation} 
which is the multiplication by $-1$ on $\bF_3^2$. We define the {\bf Burkhardt space} $\bP_\rB \cong \bP^4$ and {\bf Maschke space} $\bP_\rM \cong \bP^3$ as the $+1$ and $-1$ eigenspaces, respectively, of $\iota$. More explicitly, 
\begin{align*}
\bP_\rB &: z_2 - z_3 = z_4 - z_7 = z_5 - z_9 = z_6 - z_8 = 0,\\
\bP_\rM &: z_1 = z_2 + z_3 = z_4 + z_7 = z_5 + z_9 = z_6 + z_8 = 0.
\end{align*}

A general semisimple element is conjugate to an element of the form 
\begin{align} \label{eqn:w39general}
c_1 h_1 + c_2 h_2 + c_3 h_3 + c_4 h_4
\end{align}
for $c_1, c_2, c_3, c_4 \in K$, and two such elements are conjugate under $G$ if and only if they are conjugate to one another via the complex reflection group $W = \rG_{32}$. 

The Cartan subspace $\fh$ has the following alternative description. We have an action of the Heisenberg group $H_{3,2}$ (see \S\ref{sec:heisenberg}) on $A$, and $\fh$ is the invariant subspace $(\bigwedge^3 A)^{H_{3,2}}$. From this, we also get the action of $N(H_{3,2}) / H_{3,2} \cong \Sp_4(\bF_3)$ on $\fh$. If we also add the scalar matrices corresponding to cube roots of unity, then we get an action of the group $\rG_{32}$ on the Cartan subspace. Here are two generators for $\rG_{32}$:
\[
\mu = -\omega \begin{pmatrix}
1& 0& 0&  0\\
0& {-1}& 0&  0\\
0&  0&  0&  {-1}\\
0&  0& {-1}&  0
\end{pmatrix},
\qquad
\nu = \frac{1}{\omega - \omega^2}
\begin{pmatrix}0&  1&  1&   1\\
      0&     1&      -\omega-1&      \omega\\
      0&      {-1}&      {-\omega}&      \omega+1\\
      -2 \omega-1&      0&      0&      0
\end{pmatrix}
\]
(here $\omega$ is a fixed primitive cube root of unity).

We now state the main property of this representation and its connections to geometry. This result was essentially known to Burkhardt \cite{burkhardt} (see also \cite[\S 4]{dl}), but we would like to fit it into the context of the Vinberg representation $\bigwedge^3 A$, and this direct relationship does not seem to be written down anywhere.

\begin{theorem} \label{thm:M23minus}
Let $\fh^\circ$ be the complement of the $40$ reflection hyperplanes in $\fh$. Then $\bP(\fh^\circ)$ is in natural bijection with the set of smooth genus $2$ curves together with a marked Weierstrass point and full level $3$ structure on its Jacobian variety. Furthermore, $\rG_{32} / (\bZ/3 \times \{\pm I_4\}) \cong \bP\Sp_4(\bF_3)$ acts freely on $\bP(\fh^\circ)$ and taking the quotient $\bP(\fh^\circ) / \bP\Sp_4(\bF_3)$ has the effect of forgetting the level structure.
\end{theorem}

\begin{proof}
In the next section, given a vector in $\fh^\circ$, we construct an Abelian surface embedded in $\bP^8$ of the form $(\Jac(C), 3\Theta)$, where $C$ is a smooth genus 2 curve, together with a marked Weierstrass point on $C$. The embedding respects the Heisenberg group action, so gives a full level $3$ structure. The embedding is equivalent to the choice of $5$ scalars, which are determined by the functions $\gamma_1, \dots, \gamma_5$ in \S\ref{ss:burkhardt}, and the image of the map $\rG_{32} \to \bP\GL_5(K)$ given by the action of $\rG_{32}$ on the functions $[\gamma_1:  \dots : \gamma_5]$ is precisely $\bP\Sp_4(\bF_3)$, so we get all level structures on a given Abelian surface. This construction only depends on the vector up to scaling, and taking another vector in its $W$-orbit gives the same Abelian surface (except with a different level structure). Proposition~\ref{prop:w39disc} says that ``general'' is the same as belonging to $\fh^\circ$. In Remark~\ref{rmk:w39section} we explain how to reverse this construction.
\end{proof}

\begin{remark}
Given an element $\bv \in \bigwedge^3 A$, we can form a map $\bigwedge^3 A \to \bigwedge^6 A$ which is multiplication by $\bv$. Since the Lie bracket in $\fe_8 = \fsl(A) \oplus \bigwedge^3 A \oplus \bigwedge^6 A$ is equivariant with respect to the subalgebra $\fsl(A)$, this map agrees with multiplying by $\bv$ considered as an element of $\fe_8$ (up to an overall choice of scalar factor). A general element $\bv$ is semisimple and belongs to a unique Cartan subspace, so the kernel of this multiplication map contains this Cartan subspace. One can find elements $\bv$ (by randomly sampling with a computer, for example) such that the kernel has dimension exactly $4$. By semicontinuity, we see that the kernel must always have dimension at least $4$.
\end{remark}

\subsection{Coble cubics and Abelian surfaces}
\label{sec:coblecubics}

There is a natural $G$-equivariant isomorphism
\[
\bigwedge^3 A = \rH^0(\bP(A^*); \bigwedge^2 \cQ^* \otimes \cO(1))
\]
where $\cQ$ is the hyperplane bundle on $\bP(A^*)$. We denote $\bP(A^*)$ by $\bP^8$. We have a natural inclusion
\[
\bigwedge^2 \cQ^* \otimes \cO(1) \subset \bigwedge^2 A \otimes \cO(1),
\]
so a vector $\bv \in \bigwedge^3 A$ may be interpreted as a $9 \times 9$ skew-symmetric matrix $\Phi_\bv(z)$ with linear entries in $\bP^8$. The map on sections
\[
\bigwedge^3 A \to \bigwedge^2 A \otimes A
\]
is the usual comultiplication map on the exterior algebra.

By \cite[\S 5]{gsw}, we know that for a general choice of vector $\bv \in \bigwedge^3 A$, the ideals of $8 \times 8$ Pfaffians and $6 \times 6$ Pfaffians of $\Phi_\bv(z)$ cut out a cubic hypersurface $Y_\bv$ and a $(3,3)$-polarized Abelian surface $X_\bv$, respectively. More precisely, the $i$th principal $8 \times 8$ Pfaffian is divisible by $z_i$, and the remaining term is the cubic equation of the hypersurface (we could also phrase this in terms of saturating ideals). Furthermore, this cubic hypersurface $Y_\bv$ is singular precisely along $X_\bv$, and is known as the {\bf Coble cubic} (see \cite{beauville}) of $X_\bv$. In general, we know (from the chain rule) that the partial derivatives of the Coble cubic vanish on the set cut out by the $6 \times 6$ Pfaffians. 

If $\bv \in \fh$, then the precise meaning of ``general'' is that the discriminant $\Delta$, which is the product of the four polynomials in \eqref{eqn:discw39}, is nonzero. 

The following theorem will not be used in the remainder of the text, but gives a convenient way to recover a genus $2$ curve $C_\bv$ (in its tri-canonical embedding) whose Jacobian is $X_\bv$. When restricting to $\bv \in \fh$, this gives a family of smooth genus $2$ curves over the base $\bP^3 \setminus \{\Delta=0\}$ whose Jacobian is the universal family of Abelian surfaces that we have just discussed. The construction is given after the proof of the theorem, see also Theorem~\ref{thm:summary} for a summary.

\begin{theorem} \label{thm:fanoP4}
Pick $\bv \in \bigwedge^3 A$ so that $X_\bv$ is smooth.
\begin{compactenum}[\rm (1)]
\item Given a $4$-dimensional linear subvariety $\bP^4 \subset Y_\bv$, the intersection $\bP^4 \cap X_\bv$ is a theta-divisor of $X_\bv$. The correspondence $\bP^4 \mapsto \bP^4 \cap X_\bv$ is an isomorphism between the Fano variety of $4$-dimensional subspaces of $Y_\bv$ and the variety $X_\bv^\vee = \Pic_1(X_\bv)$.

\item Given $x \in X_\bv$, set $\cK_x = \ker(\Phi_\bv(x))$. The map $X_\bv \to \Gr(5,A^*)$ given by $x \mapsto \cK_x$ is an isomorphism from $X_\bv$ onto the Fano variety of $4$-dimensional subspaces of $Y_\bv$. 
\end{compactenum}
In particular, the map $X_{\bv}\to X_{\bv}^{\vee}$ given by $x \mapsto \bP(\cK_x) \cap X_\bv$ is an isomorphism.
\end{theorem} 

\begin{proof} 
We first prove (1). Pick $W \cong \bP^4 \subset Y_\bv$. Consider the exact sequence of conormal bundles
\[
0 \to \rN^{*}_{Y_{\bv}/\bP^8}|_{W} \to \rN^{*}_{W/\bP^8} \to \rN^{*}_{W/Y_{\bv}} \to 0.
\]
We have $\rN^{*}_{W/\bP^8}=B\otimes\cO_{W}(-1)$, where $B$ is the $4$-dimensional subspace of $A$ consisting of the linear equations of $W$ in $\bP^8$, and $\rN^{*}_{Y_{\bv}/\bP^8} = \cO_{Y_{\bv}}(-3)$. The set $\sigma:=X_{\bv}\cap W$ of singular points of $Y_{\bv}$ lying on $W$ is scheme-theoretically the zero locus in $W$ of the transpose of the left arrow, i.e., the intersection of four quadrics in $W$. We claim that $\sigma$ is a curve of (arithmetic) genus $2$. For now we only know that $\dim(\sigma) \le 1$ since $\sigma \ne X_\bv$.

Assume that $\dim(\sigma)=0$. Then it is of degree $16$. Consider the ``projection with vertex $W$'' 
\[
p \colon \bP^8 \dashrightarrow \bP(B^*), \qquad \ell \mapsto \ell+\tilde{W}
\]
where $W$ is the projectivization of $\tilde{W} \subset A^*$ and $\ell \subset A$ is a line not contained in $\tilde{W}$, and we have made the identification $B^* = A^* / \tilde{W}$. Let $\tilde{X}_\bv$ be the blow-up of $X_\bv$ at $\sigma$. Then $p$ induces a map $\tilde{X}_\bv \to \bP(B^*)$, and we denote the image by $X'_\bv$.
Since $\deg(X_{\bv})=18$, we see that $\deg(X'_\bv) \cdot \deg(p|_{\tilde{X}_\bv})=2$. Since $X_{\bv}$ linearly spans $\bP^8$ we cannot have $\deg(X'_\bv)=1$, and so $p$ induces a birational map from $X_\bv$ to an irreducible quadric $X'_\bv$. This is impossible since $X'_\bv$ is rational and $X_\bv$ is not.

Therefore $\sigma$ contains, as a component, a divisor $C\subset X_{\bv}$, and a residual zero-dimensional scheme. Since $\cO_{X_{\bv}}(1)$ is a $(3,3)$-polarization, its class in $\Pic(X_{\bv})$ is divisible by $3$, so the degree of any component of $C$ is divisible by $3$. On the other hand, $\sigma$ is the intersection of four quadrics in $W$ and its dimension is $1$. The complete intersection of three general quadrics among the four is a curve of degree $8$, so either $\deg(C)=3$ or $\deg(C)=6$.

Assume that $C$ has a component $C'$ (necessarily reduced and irreducible) of degree $3$. The self-intersection of $C'$ on $X_{\bv}$ is $2p_a(C')-2$, where $p_a(C')$ is the arithmetic genus of $C'$, and this self-intersection is nonnegative since $C'$ is mobile enough (by translations). So $C'$ is a plane cubic, and it cuts a general theta-divisor in one point  since $(3\theta\mid C')=3$. Since there is one theta-divisor through two points of $C'$ we see that $C'$ is a component of this theta-divisor, the other component being another plane cubic $C''$ cutting $C$ in one point (since $p_a(C\cup C') = 2$). This easily implies that $X_{\bv}$ is the product of these two (necessarily smooth) curves as a polarized variety, which contradicts \cite[Proposition 5.6]{gsw}.

So we conclude that $C$ is an irreducible and generically reduced curve of degree $6$. The next lemma implies that $C$ is a theta-divisor of $X_{\bv}$.

\begin{lemma}
Let $B$ be a $4$-dimensional vector space of quadrics in $W$, defining a codimension $3$ subscheme $\sigma$ whose top-dimensional component is an irreducible and generically reduced curve $C$ of degree $6$. Then $C = \sigma$ is a curve of arithmetic genus $2$. 
\end{lemma}

\begin{proof} 
Let $\tilde{W}$ be the blow-up of $\sigma$ in $W$ and denote the blow-up map $\pi \colon \tilde{W} \to W$. The inclusion $B\subset \rH^{0}(\cI_{\sigma/W}(2))$ gives a dominant morphism $f \colon \tilde{W}\to \bP(B^*)$ whose general fiber is integral by Bertini's theorem. Let $x \in \bP(B^*)$ be a general point, and set $\Gamma = \pi(f^{-1}(x))$. Then $\Gamma$ is an integral curve and hence Cohen--Macaulay. Also, $C\cup\Gamma$ is contained in the complete intersection $V \subset \bP(B^*)$ of $3$ independent quadrics in the hyperplane of $B$ represented by $x$, and since the complement contains no curve it must be empty, i.e., $C$ and $\Gamma$ are linked by these quadrics. 

Then $\cI_{C/V} = \cH om_{\cO_V}(\cO_\Gamma, \cO_V) = \omega_\Gamma(-1)$, where the first equality is \cite[Theorem 21.23]{eisenbud} and the second equality is \cite[Theorem 21.15]{eisenbud} combined with $\omega_V = \cO_V(1)$. So we have the liaison exact sequence 
\[
0 \to \omega_{\Gamma}(-1) \to \cO_{V} \to \cO_{C} \to 0.
\]
The Hilbert polynomials of these sheaves are $2t-3+p_a(\Gamma)$, $8t-4$, and $6t+1-p_a(C)$, respectively, from which we conclude that 
\[
p_a(C) = p_a(\Gamma) + 2.
\]
Since $\Gamma$ is integral, it follows that $p_a(\Gamma) \ge p_a(\tilde{\Gamma}) \ge 0$ ($\tilde{\Gamma}$ is the normalization of $\Gamma$), and hence $p_a(C) \ge 2$.
Furthermore, since $\Gamma$ is Cohen--Macaulay, the same is true for $C$ \cite[Theorem 21.23(b)]{eisenbud}. We have assumed that $C$ is irreducible and generically reduced, so this implies that $C$ is an integral curve. By the Castelnuovo bound, $2$ is the maximum possible genus of a non-degenerate integral curve of degree $6$ in $W \cong \bP^4$, and such curves are scheme-theoretically cut out by $4$ quadrics, so we conclude that $\sigma = C$ and we are done.
\end{proof}

Conversely, if $C$ is a theta-divisor on $X_{\bv}$, it is a smooth curve of degree $6$ and genus $2$ normally embedded in a $4$-dimensional linear subvariety $W \cong \bP^4$ of $\bP^8$. The secant variety of $C$ is a hypersurface of $W$ of degree $8$, since the plane projection of $C$ from a general line of $W$ has eight double points. All of the partial derivatives of the equation of $Y_\bv$ vanish on $X_\bv$, and hence $C$, so this hypersurface is contained in $Y_{\bv}$. By degree considerations, we know that $W$ itself is contained in $Y_{\bv}$, and this finishes the proof of (1).

Now we prove (2). By definition, $[x] \in X_{\bv}$ if and only if $\rank \Phi_\bv(x) = 4$ and $[x] \in Y_\bv$ if and only if $\rank \Phi_\bv(x) \le 6$. We can also write $\Phi_\bv(x)$ as the contraction $d_x \bv$ (recall that $x \in A^*$ and $\bv \in \bigwedge^3 A$, so $d_x \bv \in \bigwedge^2 A$). For any $y\in \cK_{x}$ we have $d_{x\wedge y}\bv=0$, so the kernel of $d_{y}\bv = \Phi_\bv(y)$ contains $x$. If $y \ne x$, this implies that $\dim \ker (d_y \bv) \geq 2$ (and hence $\ge 3$ by skew-symmetry), and so $[y] \in Y_{\bv}$ (and if $y=x$, then we already know that $[x] \in X_\bv \subset Y_\bv$). We conclude that $\bP(\cK_{x}) \subset Y_{\bv}$. Let us suppose now that $\cK_{x}=\cK_{x'}$ for two distinct points $[x] \ne [x']$ of $X_{\bv}$: then for any linear combination $x''$ of $x$ and $x'$ and any $y$ such that $y \in \cK_{x}$, we have $d_{x''\wedge y}\omega=0$. So the line joining the two points $[x]$ and $[x']$ lies on $X_{\bv}$. This is a contradiction since Abelian varieties do not contain rational curves \cite[Proposition 4.9.5]{birkenhake}. Now using (1), we see that the map $X_\bv \to X_\bv^\vee$ given by $x \mapsto \bP(\cK_x) \cap X_\bv$ is injective and hence must also be surjective since both source and target are irreducible surfaces. Since the target is smooth, this implies that the map is an isomorphism.
\end{proof}

Now we specialize to the case $\bv \in \fh$. In this case, we write $\bv = \bc = (c_1,c_2,c_3,c_4)$ using the coordinates \eqref{eqn:w39general}, and the skew-symmetric matrix takes on the following form
\begin{align} \label{eqn:phimatrix}
\Phi_\bc(z) = \begin{pmatrix}
  0& {-c_1 z_{3}}& c_1 z_{2}& {-c_2 z_{7}}& {-c_3 z_{9}}& {-c_4
    z_{8}}& c_2 z_{4}& c_4 z_{6}&  c_3 z_{5}\\
  c_1 z_{3}& 0& {-c_1 z_{1}}& {-c_4 z_{9}}& {-c_2 z_{8}}& {-c_3
    z_{7}}& c_3 z_{6}& c_2 z_{5}&  c_4 z_{4}\\  
{-c_1 z_{2}}& c_1 z_{1}& 0& {-c_3 z_{8}}& {-c_4 z_{7}}& {-c_2
    z_{9}}& c_4 z_{5}& c_3 z_{4}&  c_2 z_{6}\\
  c_2 z_{7}& c_4 z_{9}& c_3 z_{8}& 0& {-c_1 z_{6}}& c_1 z_{5}& {-c_2
    z_{1}}& {-c_3 z_{3}}&  {-c_4 z_{2}}\\
  c_3 z_{9}& c_2 z_{8}& c_4 z_{7}& c_1 z_{6}& 0& {-c_1 z_{4}}& {-c_4
    z_{3}}& {-c_2 z_{2}}&  {-c_3 z_{1}}\\
  c_4 z_{8}& c_3 z_{7}& c_2 z_{9}& {-c_1 z_{5}}& c_1 z_{4}& 0& {-c_3
    z_{2}}& {-c_4 z_{1}}&  {-c_2 z_{3}}\\
  {-c_2 z_{4}}& {-c_3 z_{6}}& {-c_4 z_{5}}& c_2 z_{1}& c_4 z_{3}& c_3
  z_{2}& 0& {-c_1 z_{9}}&  c_1 z_{8}\\
  {-c_4 z_{6}}& {-c_2 z_{5}}& {-c_3 z_{4}}& c_3 z_{3}& c_2 z_{2}& c_4
  z_{1}& c_1 z_{9}& 0&  {-c_1 z_{7}}\\
  {-c_3 z_{5}}& {-c_4 z_{4}}& {-c_2 z_{6}}& c_4 z_{2}& c_3 z_{1}& c_2
  z_{3}& {-c_1 z_{8}}& c_1 z_{7}& 0
\end{pmatrix}
\end{align}

The equation for the Coble cubic $Y_\bc$ is
\small\begin{equation} \label{eqn:coblecubic}
\begin{split}
c_1c_2c_3c_4 \sum_{i=1}^9 z_i^3 
- c_1(c_2^3+c_3^3+c_4^3)(z_1z_2z_3 + z_4z_5z_6 + z_7z_8z_9)
+ c_2(c_1^3+c_3^3-c_4^3)(z_1z_4z_7 + z_2z_5z_8 + z_3z_6z_9)\\
+ c_3(c_1^3-c_2^3+c_4^3)(z_1z_5z_9 + z_2z_6z_7 + z_3z_4z_8)
+ c_4(c_1^3+c_2^3-c_3^3)(z_1z_6z_8 + z_2z_4z_9 + z_3z_5z_7).
\end{split}
\end{equation}
\normalsize We call the ideal generated by the partial derivatives of \eqref{eqn:coblecubic} the Jacobian ideal:
\small\begin{align} \label{eqn:jacobian}
\begin{matrix}
3 c_1 c_2 c_3 c_4 z_{1}^{2}-
c_1(c_2^{3}+ c_3^{3}+ c_4^{3}) z_{2} z_{3}+
c_2(c_1^{3}+ c_3^{3}- c_4^{3}) z_{4} z_{7}+
c_3(c_1^{3}-c_2^{3}+c_4^{3}) z_{5} z_{9}+
c_4(c_1^{3}+c_2^{3}-c_3^{3}) z_{6} z_{8}\\
3 c_1 c_2 c_3 c_4 z_2^{2}-
c_1(c_2^{3}+ c_3^{3}+ c_4^{3}) z_{1} z_{3}+
c_2(c_1^{3}+ c_3^{3}- c_4^{3}) z_{5} z_{8}+
c_3(c_1^{3}-c_2^{3}+c_4^{3}) z_{6} z_{7}+
c_4(c_1^{3}+c_2^{3}-c_3^{3}) z_{4} z_{9}\\
3 c_1 c_2 c_3 c_4 z_{3}^{2}-
c_1(c_2^{3}+ c_3^{3}+ c_4^{3}) z_{1} z_{2}+
c_2(c_1^{3}+ c_3^{3}- c_4^{3}) z_{6} z_{9}+
c_3(c_1^{3}-c_2^{3}+c_4^{3}) z_{4} z_{8}+
c_4(c_1^{3}+c_2^{3}-c_3^{3}) z_{5} z_{7}\\
3 c_1 c_2 c_3 c_4 z_{4}^{2}-
c_1(c_2^{3}+ c_3^{3}+ c_4^{3}) z_{5} z_{6}+
c_2(c_1^{3}+ c_3^{3}- c_4^{3}) z_{1} z_{7}+
c_3(c_1^{3}-c_2^{3}+c_4^{3}) z_{3} z_{8}+
c_4(c_1^{3}+c_2^{3}-c_3^{3}) z_{2} z_{9}\\
3 c_1 c_2 c_3 c_4 z_{5}^{2}-
c_1(c_2^{3}+ c_3^{3}+ c_4^{3}) z_{4} z_{6}+
c_2(c_1^{3}+ c_3^{3}- c_4^{3}) z_{2} z_{8}+
c_3(c_1^{3}-c_2^{3}+c_4^{3}) z_{1} z_{9}+
c_4(c_1^{3}+c_2^{3}-c_3^{3}) z_{3} z_{7}\\
3 c_1 c_2 c_3 c_4 z_{6}^{2}-
c_1(c_2^{3}+ c_3^{3}+ c_4^{3}) z_{4} z_{5}+
c_2(c_1^{3}+ c_3^{3}- c_4^{3}) z_{3} z_{9}+
c_3(c_1^{3}-c_2^{3}+c_4^{3}) z_{2} z_{7}+
c_4(c_1^{3}+c_2^{3}-c_3^{3}) z_{1} z_{8}\\
3 c_1 c_2 c_3 c_4 z_{7}^{2}-
c_1(c_2^{3}+ c_3^{3}+ c_4^{3}) z_{8} z_{9}+
c_2(c_1^{3}+ c_3^{3}- c_4^{3}) z_{1} z_{4}+
c_3(c_1^{3}-c_2^{3}+c_4^{3}) z_{2} z_{6}+
c_4(c_1^{3}+c_2^{3}-c_3^{3}) z_{3} z_{5}\\
3 c_1 c_2 c_3 c_4 z_{8}^{2}-
c_1(c_2^{3}+ c_3^{3}+ c_4^{3}) z_{7} z_{9}+
c_2(c_1^{3}+ c_3^{3}- c_4^{3}) z_{2} z_{5}+
c_3(c_1^{3}-c_2^{3}+c_4^{3}) z_{3} z_{4}+
c_4(c_1^{3}+c_2^{3}-c_3^{3}) z_{1} z_{6}\\
3 c_1 c_2 c_3 c_4 z_{9}^{2}-
c_1(c_2^{3}+ c_3^{3}+ c_4^{3}) z_{7} z_{8}+
c_2(c_1^{3}+ c_3^{3}- c_4^{3}) z_{3} z_{6}+
c_3(c_1^{3}-c_2^{3}+c_4^{3}) z_{1} z_{5}+
c_4(c_1^{3}+c_2^{3}-c_3^{3}) z_{2} z_{4}
\end{matrix}
\end{align}

\normalsize

In particular, we get universal families of Coble cubic hypersurfaces and Abelian surfaces over $\bP^3 \setminus \{\Delta = 0\}$, where the coordinates on the base are $[c_1:c_2:c_3:c_4]$.
If we change to coordinates fixed by the involution $\iota$ \eqref{eqn:iota} and restrict to the Maschke space $\bP_\rM$, then $\Phi_\bc(z)$ from \eqref{eqn:phimatrix} takes on the block matrix form
\[
\footnotesize\begin{pmatrix}0& {-c_1 z_{3}}& {-c_2 z_{7}}&
  {-c_3 z_{9}}& {-c_4 z_{8}}& 0& 0& 0&  0\\
  c_1 z_{3}& 0& -c_3 z_{8}-c_4 z_{9}& -c_4 z_{7}-c_2 z_{8}&
  -c_3 z_{7}-c_2  z_{9}& 0& 0& 0&  0\\
  c_2 z_{7}& c_3 z_{8}+c_4 z_{9}& 0& c_4 z_{3}+c_1 z_{8}&
  -c_3 z_{3}-c_1 z_{9}& 0&  0& 0&  0\\
  c_3 z_{9}& c_4 z_{7}+c_2 z_{8}& -c_4 z_{3}-c_1 z_{8}& 0&
  c_2 z_{3}+c_1 z_{7}&  0& 0& 0&  0\\
  c_4 z_{8}& c_3 z_{7}+c_2 z_{9}& c_3 z_{3}+c_1 z_{9}& -c_2
  z_{3}-c_1 z_{7}& 0&  0& 0& 0&  0\\
  0& 0& 0& 0& 0& 0& c_3 z_{8}-c_4 z_{9}& c_4
  z_{7}-c_2 z_{8}&  c_2 z_{9}-c_3 z_{7}\\
  0& 0& 0& 0& 0& c_4 z_{9}-c_3 z_{8}& 0& 
c_1 z_{8}-c_4z_{3}&  c_3 z_{3}-c_1 z_{9}\\
  0& 0& 0& 0& 0& c_2 z_{8}-c_4 z_{7}& c_4
  z_{3}-c_1 z_{8}& 0&  c_1 z_{7}-c_2 z_{3}\\
  0& 0& 0& 0& 0& c_3 z_{7}-c_2 z_{9}& c_1 z_{9}-c_3 z_{3}&
  c_2 z_{3}-c_1 z_{7}& 0
       \end{pmatrix}.
\]
The bottom right $4 \times 4$ matrix has rank $2$, and the linear entries define the point 
\begin{align} \label{eqn:universalsection}
z_\bc = [0: -c_1 : c_1: -c_2: -c_3: -c_4: c_2: c_4: c_3]
\end{align}
on the Abelian surface $X_\bc$. Using Theorem~\ref{thm:fanoP4}, this distinguished point gives us a distinguished $\bP^4$ in the Coble cubic $Y_\bc$ (and hence a family in the universal Coble cubic) by taking the kernel of the matrix
\begin{align} \label{eqn:matrixc}
\Phi_\bc(z_\bc) = \begin{pmatrix}
0 & -c_1^2 & -c_1^2 & -c_2^2 & -c_3^2 & -c_4^2 & -c_2^2 & -c_4^2 & -c_3^2\\
c_1^2 & 0 & 0 & -c_3c_4 & -c_2c_4 & -c_2c_3 & -c_3c_4 & -c_2c_3 & -c_2c_4\\
c_1^2 & 0 & 0 & -c_3c_4 & -c_2c_4 & -c_2c_3 & -c_3c_4 & -c_2c_3 & -c_2c_4\\
c_2^2 & c_3c_4 & c_3c_4 & 0 & c_1c_4 & -c_1c_3 & 0 & -c_1c_3 & c_1c_4\\
c_3^2 & c_2c_4 & c_2c_4 & -c_1c_4 & 0 & c_1c_2 & -c_1c_4 & c_1c_2 & 0\\
c_4^2 & c_2c_3 & c_2c_3 & c_1c_3 & -c_1c_2 & 0 & c_1c_3 & 0 & -c_1c_2\\
c_2^2 & c_3c_4 & c_3c_4 & 0 & c_1c_4 & -c_1c_3 & 0 & -c_1c_3 & c_1c_4\\
c_4^2 & c_2c_3 & c_2c_3 & c_1c_3 & -c_1c_2 & 0 & c_1c_3 & 0 & -c_1c_2\\
c_3^2 & c_2c_4 & c_2c_4 & -c_1c_4 & 0 & c_1c_2 & -c_1c_4 & c_1c_2 & 0
\end{pmatrix}.
\end{align}
We note that $\bP_{\rM} \subset \ker \Phi_\bc(z_\bc)$. This matrix generically has rank $4$, and the ideal of $4 \times 4$ Pfaffians is the ideal generated by the coefficients of the Coble cubic. When it has rank $4$, the kernel is cut out by the following $4$ linear equations (we are assuming that all $c_i$ are nonzero):
\small 
\begin{equation} \label{eqn:universalP4}
\begin{split}
(c_2^3+c_3^3+c_4^3)z_1+3c_2c_3c_4(z_2+z_3) = 0, \qquad
(-c_1^3-c_3^3+c_4^3)z_1+3c_1c_3c_4(z_4+z_7) = 0,\\
(-c_1^3-c_2^3+c_3^3)z_1+3c_1c_2c_3(z_6+z_8) = 0, \qquad
(-c_1^3+c_2^3-c_4^3)z_1+3c_1c_2c_4(z_5+z_9) = 0.
\end{split}
\end{equation}
\normalsize
Again by Theorem~\ref{thm:fanoP4}, the intersection of this $\bP^4$ with the Abelian surface $X_\bc$ gives the tri-canonical embedding of a genus $2$ curve $C_\bc$. Now we record what we have learned so far.

\begin{theorem} \label{thm:summary}
Set $S = \bP^3 \setminus \{\Delta=0\}$.
\begin{compactenum}[\rm (a)]
\item The $8 \times 8$ Pfaffians of the matrix $\Phi_\bc(z)$ in \eqref{eqn:phimatrix} define a flat family of cubic hypersurfaces $\cY \subset \bP^8_S$ over $S$.
\item The singular locus of this family is a flat family of Abelian surfaces $\cX \subset \bP^8_S$ over $S$ equipped with indecomposable $(3,3)$-polarizations. The identity point is given by the universal section \eqref{eqn:universalsection}. This family can also be defined by the $6 \times 6$ Pfaffians of the matrix $\Phi_\bc(z)$.
\item The equations \eqref{eqn:universalP4} define a $\bP^4$-bundle $\cP \subset \bP^8_S$ over $S$ which is a subvariety of $\cY$.
\item The intersection $\cC = \cP \cap \cX$ is a flat family of smooth genus $2$ curves over $S$, and $(\cX, \cO_{\cX}(1)) = (\Jac(\cC), 3\Theta)$.
\end{compactenum}
\end{theorem}

\begin{remark}
The intersection $\bP_\rM \cap X_\bc$ consists of 6 points, and the intersection $\bP_\rB \cap X_\bc$ consists of $10$ points \cite[\S 4.3]{dl}. The union of these points is the $2$-torsion subgroup $X_\bc[2]$. Also, $\bP_\rM \cong \bP^3$ and the $6$ points are in linearly general position, so there is a unique rational normal cubic through them. The double cover of this $\bP^1$ ramified at these $6$ points is a genus $2$ curve, which is $C = X_\bc \cap \bP^4$, where $\bP^4$ is defined by \eqref{eqn:universalP4}. The point \eqref{eqn:universalsection} is a marked Weierstrass point on $C$. 

Here's how to see this explicitly: recall the involution $\iota$ on $\bP^8$ defined in \S\ref{ss:coblecubics} which extends the usual involution $x \mapsto -x$ on the Abelian surface $X_\bc$ with identity point $z_\bc$ \eqref{eqn:universalsection}. The $\bP^4$ defined in \eqref{eqn:universalP4} is invariant under $\iota$ and hence the same is true for $C$. The map $x \mapsto x - \iota(x)$ (here we are interpreting $x$ as a vector in $A^* \cong K^9$ and using its linear structure; this is not to be confused with negation on $X_\bc$) gives a map $C \to \bP_{\rM}$ which contains the $6$ points $X_\bc \cap \bP_{\rM}$ (they are fixed points for this map and $\bP_{\rM} \subset \bP^4$) and hence its image is the rational normal curve as claimed.

Here is a way to find the equations for this twisted cubic. The ideal of the $6$ points in $\bP_\rM$ is defined by $4$ quadrics. Its Betti table is
\small \begin{Verbatim}[samepage=true]
       0 1 2 3
total: 1 4 5 2
    0: 1 . . .
    1: . 4 2 .
    2: . . 3 2
\end{Verbatim}
\normalsize
This looks like an Eagon--Northcott complex tensored with a Koszul complex of length $1$ on a quadric. We can make this observation useful as follows: in the third differential, we see a $2 \times 3$ matrix of linear forms, and the ideal of $2 \times 2$ minors of this matrix vanishes on the $6$ points and defines the curve.
\end{remark}

\begin{remark} \label{rmk:33sat}
The Jacobian ideal \eqref{eqn:jacobian} is not radical. Indeed, for generic choices of \eqref{eqn:w39general}, this ideal scheme-theoretically defines an Abelian surface, but there are 3 more cubics in its saturation. (This had previously been shown by Barth \cite{barth}.) These cubics can be obtained by taking the $6 \times 6$ Pfaffians of \eqref{eqn:phimatrix}. 

Now consider $c_1, \dots, c_4$ as projective coordinates on $\bP^3$ and consider the subvariety $\overline{\cX}$ of $\bP^3 \times \bP^8$ defined by the Jacobian ideal, which is generated by the $9$ bidegree $(4,2)$-equations \eqref{eqn:jacobian} and the $84$ bidegree $(3,3)$ equations coming from the $6 \times 6$ Pfaffians of \eqref{eqn:phimatrix}. We might conjecture that these $93$ equations define the radical ideal of $\cX$ (they are all linearly independent). This subvariety $\overline{\cX}$ is like the ``universal $(3,3)$ Abelian surface''. Compare this to \cite[Conjecture 8.1]{univkummer}, which deals with the analogous situation of Coble quartics, rather than Coble cubics. Similarly, using the ``universal $\bP^4$'' $\overline{\cP}$ defined as the projectivization of the kernel of \eqref{eqn:matrixc}, we can ask about the prime ideal of the ``universal genus 2 curve'' $\overline{\cC} = \overline{\cX} \cap \overline{\cP}$ in its tri-canonical embedding. Technically, $\overline{\cP}$ is not a $\bP^4$-bundle, since the matrix \eqref{eqn:matrixc} drops to rank $2$ over the points $\bc$ where the coefficients of the Coble cubic vanish.
\end{remark}

\begin{remark}
If we set $z_1=z_2=z_3$, $z_4=z_5=z_6$, $z_7=z_8=z_9$, then the equation for the Coble cubic becomes the Hesse normal form for a plane cubic. Similarly, we can get three more plane cubics by considering different affine lines in \eqref{eqn:affineplane}. 
\end{remark}

\begin{remark}
The Jacobian variety of a curve $C$ naturally has a modular interpretation as the moduli space of degree $0$ line bundles on $C$. The Coble cubic is also related to a moduli space: it is shown in \cite{nguyen} and \cite{ortega} that the moduli space of semistable principal $\SL_2$-bundles on a genus $2$ curve $C$ can be exhibited as a double cover of $\bP^8$, and the branch locus is the projective dual of the Coble cubic of $C$.
\end{remark}

\subsection{The reflection group $\rG_{32}$}

Recall that a {\bf flat} of a hyperplane arrangement is just an intersection of some of the hyperplanes. For the reflection arrangement of a complex reflection group $\Gamma$, there is a natural action of $\Gamma$ on the flats. For $\Gamma = \rG_{32}$, we list the orbits of the flats in Table~\ref{tab:G32flats}. The column ``codimension'' gives the codimension of the flat, ``size'' refers to the number of flats in the orbit, ``parabolic subgroup'' is the isomorphism type of the pointwise stabilizer of the flat, and the last column on ``equations'' gives the equations of one flat in the orbit. We note that $\rG_4 \cong \SL_2(\bF_3)$ and that $\rG_{25}$, modulo its $3$-element center, is the Hessian group, which is the automorphism group of the Hessian pencil of plane cubics, see \cite[\S 4]{artebanidolgachev} for more details. Again, we are using the notation of \cite{shephard} to index complex reflection groups. The numbering of the flats follows \cite{ev}.

\begin{table}
\begin{tabular}{l|l|l|l|l}
Family \# & Codimension & Size & Parabolic subgroup & Equations of a representative flat \\
\hline
1 & 0 & & & Generic\\
\hline
2 & 1 & 40 & $\bZ/3$ & $c_4$\\
\hline
3 & 2 & 240 & $\bZ/3 \times \bZ/3$ & $c_3,c_4$ \\
4 & 2 & 90 & $\rG_4$ & $c_2+c_3,c_4$\\
\hline
5 & 3 & 360 & $\bZ/3 \times \rG_4$ & $c_1, c_2+c_3, c_4$\\
6 & 3 & 40 & $\rG_{25}$ & $c_2,c_3,c_4$\\
\end{tabular}
\caption{The flats of the $\rG_{32}$ reflection arrangement.}
\label{tab:G32flats}
\end{table}

The discriminant $\Delta(\bc)$ of $\rG_{32}$ is the defining equation for the union of the $40$ hyperplanes in its reflection arrangement, and is the product of the following $4$ polynomials 
\begin{equation} \label{eqn:discw39}
\begin{split}
c_1((c_2^3+c_3^3+c_4^3)^3 - (3c_2c_3c_4)^3), \qquad 
c_2((c_1^3+c_3^3-c_4^3)^3 + (3c_1c_3c_4)^3),\\
c_3((c_1^3-c_2^3+c_4^3)^3 + (3c_1c_2c_4)^3), \qquad 
c_4((c_1^3+c_2^3-c_3^3)^3 + (3c_1c_2c_3)^3).
\end{split}
\end{equation}
This differs slightly from what is in \cite{ev}. 

\begin{proposition} \label{prop:w39disc}
The discriminant $\Delta(\bc)$ is nonzero if and only if the $6 \times 6$ Pfaffians of \eqref{eqn:phimatrix} define a smooth surface. 
\end{proposition}

\begin{proof}
Over a field of characteristic $0$, this is \cite[Proposition 4.4]{dl}, and in general for a field of characteristic different from $3$, this is \cite[\S 3]{geer}.
\end{proof}

In \S\ref{sec:w39degenerations} we will study what happens when $\Delta(\bc) = 0$.

\subsection{The Burkhardt quartic} \label{ss:burkhardt}

Using \eqref{eqn:jacobian}, we define the coefficients 
\begin{align*}
\begin{array}{lll}
\gamma_1 = 3 c_1 c_2 c_3 c_4, &
\gamma_2 = -c_1(c_2^{3}+ c_3^{3}+ c_4^{3}),&
\gamma_3 = c_2(c_1^{3}+ c_3^{3}- c_4^{3}),\\ &
\gamma_4 = c_3(c_1^{3}-c_2^{3}+c_4^{3}),&
\gamma_5 = c_4(c_1^{3}+c_2^{3}-c_3^{3}).
\end{array}
\end{align*}
These are points of the {\bf Burkhardt quartic} $\cB$
\begin{align} \label{eqn:burkhardt}
\gamma_1 (\gamma_1^3 + \gamma_2^3 + \gamma_3^3 + \gamma_4^3 + \gamma_5^3) + 3\gamma_2 \gamma_3 \gamma_4 \gamma_5 = 0.
\end{align}
See \cite[Chapter 5]{hunt} for more information on this hypersurface. 

The singular locus of the Burkhardt quartic consists of $45$ points. For each one, the variety defined by the equations \eqref{eqn:jacobian} is a Segre embedding of $\bP^2 \times \bP^2$. These correspond to polarized Abelian surfaces which are products of plane cubics. More precisely, this is true for points in the exceptional divisor of the blowup of $\cB$ at these $45$ points (cf. \cite[Theorem 3.1]{geer}).

\begin{remark}
Following \cite[\S 5.1]{hunt}, we can write \eqref{eqn:burkhardt} as the fourth elementary symmetric function $\frac{1}{3}e_4(x_0, \dots, x_5)$ where
\begin{align*}
\begin{array}{ll}
x_0 = \gamma_1-\gamma_4-\omega\gamma_5, & 
x_1 = \gamma_1-\omega\gamma_4-\gamma_5,\\
x_2 = -\gamma_1+\omega\gamma_2+\omega\gamma_3, &
x_3 = -\gamma_1+\omega^2\gamma_2+\gamma_3,\\
x_4 = -\gamma_1+\gamma_2+\omega^2\gamma_3,&
x_5 = \gamma_1-\omega^2\gamma_4-\omega^2\gamma_5.
\end{array}
\end{align*}
Our choice of variables differs from the choice in \cite{hunt} by scalar multiples, so these do not match up with \cite[\S 5.1]{hunt} exactly. 
\end{remark}

The $5$ functions $\gamma_1, \dots, \gamma_5$ form an irreducible representation of $\rG_{32}$, which is a Macdonald representation \cite{macdonald} for root subsystems of type $\rA_1^{\times 4}$. They define a rational map 
\begin{align} \label{eqn:cmap}
\gamma \colon \bP^3 \dashrightarrow \cB
\end{align}
of degree $6$ \cite[\S 5.3.1]{hunt}. Its base locus consists of $40$ points. As with any Macdonald representation, this map can be factored as the product of a linear map and a monomial map. This perspective leads one to the theory of {\it tropical compactifications}, which is pursued in \cite{RSS}.

\begin{remark} \label{rmk:gorenstein}
The base locus of the rational map $\bP^3 \dashrightarrow \bP^4$ given by the coefficients of the Coble cubic is Family 6 of Table~\ref{tab:G32flats}. This ideal is a codimension $3$ Gorenstein ideal, and can be expressed as the $4 \times 4$ Pfaffians of the matrix:
\[
\psi(\bc) = \begin{pmatrix}
0 & c_1^2 & c_2^2 & c_3^2 & c_4^2\\
-c_1^2 & 0 & c_3c_4 & c_2c_4 & c_2c_3\\
-c_2^2 & -c_3c_4 & 0 & -c_1c_4 & c_1c_3\\
-c_3^2 & -c_2c_4 & c_1c_4 & 0 & -c_1c_2\\
-c_4^2 & -c_2c_3 & -c_1c_3 & c_1c_2 & 0
\end{pmatrix} 
\]

Now consider where $\bP^3 \dashrightarrow \bP^4$ fails to be an immersion, i.e., when the Jacobian matrix
\[
\begin{pmatrix}
  3 c_2 c_3 c_4&      -c_2^3-c_3^3-c_4^3&      3 c_1^2 c_2&      3 c_1^2 c_3&      3 c_1^2 c_4\\
  3 c_1 c_3 c_4&      -3 c_1 c_2^2&      c_1^3+c_3^3-c_4^3&      -3 c_2^2 c_3&      3 c_2^2 c_4\\
  3 c_1 c_2 c_4&      -3 c_1 c_3^2&      3 c_2 c_3^2&      c_1^3-c_2^3+c_4^3&      -3 c_3^2 c_4\\
  3 c_1 c_2 c_3& -3 c_1 c_4^2& -3 c_2 c_4^2& 3 c_3 c_4^2&
  c_1^3+c_2^3-c_3^3
      \end{pmatrix}
\]
fails to be injective. The maximal minors of this matrix cut out a codimension $2$ Cohen--Macaulay reduced scheme which coincides with Family 4 of Table~\ref{tab:G32flats}.
\end{remark}

\begin{proposition} \label{prop:cmapsurj}
The image of the map $\bP^3 \setminus \mathrm{Family\ 6} \to \cB$ given by \eqref{eqn:cmap} is the complement of the $\rG_{32}$-orbit of the point $[0:0:0:0:1] \in \cB$, and this orbit consists of $160$ points.
\end{proposition}

\begin{proof}
In Remark~\ref{rmk:gorenstein}, we observed that the base locus of the rational map \eqref{eqn:cmap} (Family 6) is defined by the $4 \times 4$ Pfaffians of the matrix $\psi(\bc)$. In particular, we can resolve this rational map as follows. Define
\[
\Gamma_\bc = \{ (\bc, \gamma) \in \bP^3 \times \bP^4 \mid \psi(\bc) \cdot \gamma = 0\}
\]
where here we treat $\gamma$ as a column vector of length $5$. Let $\pi_1,\pi_2$ be the two projection maps from $\Gamma_\bc$. Then $\pi_1$ is an isomorphism over $\bP^3 \setminus \text{Family 6}$ and the map $\pi_2\pi_1^{-1} \colon \bP^3 \setminus \text{Family 6} \to \cB$ agrees with \eqref{eqn:cmap}. We know that $\pi_2(\Gamma_\bc) \subset \bP^4$ is closed, so it is enough to show that for any point $\gamma \in \cB$ different from the $160$ points mentioned above, we have that $\gamma$ is in the kernel of some $\psi(\bc')$ where $\bc'$ is not in Family 6. So we may as well assume that $\psi(\bc) \cdot \gamma = 0$ where $\bc$ is in Family 6 (otherwise there is nothing to show).

Every map we have defined is equivariant with respect to $\rG_{32}$ and Family 6 forms a single orbit under $\rG_{32}$, so it is enough to prove this for a single point $\bc$ in Family 6. We choose the point $\bc = [1:0:0:0]$. Then the kernel of $\psi(\bc)$ is a $\bP^2$ cut out by $\gamma_1 = \gamma_2 = 0$. So for arbitrary $\gamma_3, \gamma_4, \gamma_5$ outside of the collection of $160$ points, we need to find a point $\bc'$ outside of Family 6 that maps to this point under \eqref{eqn:cmap}. If we set $c_1=0$, then we automatically get $\gamma_1 = \gamma_2 = 0$ and we need to solve the equations
\[
\gamma_3 = c_2(c_3^3 - c_4^3), \qquad \gamma_4 = c_3(-c_2^3 + c_4^3), \qquad \gamma_5 = c_4(c_2^3 - c_3^3).
\]
These functions are in the kernel of the following matrix:
\[
\begin{pmatrix}
c_2^2 & c_3^2 & c_4^2\\
c_3c_4 & c_2c_4 & c_2c_3
\end{pmatrix}.
\]
So as before, we can resolve the rational map defined by $\gamma_3, \gamma_4, \gamma_5$ by using kernel elements of this map. There are $12$ points $[c_2:c_3:c_4]$ for which this matrix has rank $1$, and we focus on one of them, namely the point $\tilde{\bc} = [0:1:0:0]$. In this case, the kernel is cut out by $\gamma_3 = 0$. If we set $c_2 = 0$, then we get $\gamma_3 = 0$ and then we need to solve the equations $\gamma_4 = c_3 c_4^3$ and $\gamma_5 = -c_3^3c_4$ for arbitrary values of $\gamma_4$ and $\gamma_5$. This is possible assuming that $\gamma_4 \gamma_5 \ne 0$. The points $[0:0:0:1:0]$ and $[0:0:0:0:1]$ are in $\cB$ and this suggests that they are not in the image of \eqref{eqn:cmap} (one can check that they are not). They are the only two missing points given our choice of $\bc, \tilde{\bc}$. 

There are $40 \cdot 12 / 2 = 240$ choices for the unordered pair $\bc, \tilde{\bc}$, and hence we get at most $480$ missing points. One can check that $\rG_{32}$ acts transitively on the set of unordered pairs $\bc, \tilde{\bc}$, so these missing points also form a single orbit under $\rG_{32}$ (in the above example, $[0:0:0:1:0]$ and $[0:0:0:0:1]$ are conjugate under $\rG_{32}$). A direct calculation shows that the orbit has size $160$, so we are missing $160$ points.
\end{proof}

\begin{remark}
If we study Coble quartics instead of Coble cubics (see \cite{beauville}), then the analogue of the Burkhardt quartic is called the G\"opel variety in \cite{univkummer}. It was shown by Coble \cite[\S 49]{coble2} to be cut out set-theoretically by cubics. The entire prime ideal was determined in \cite[\S 5]{univkummer}.
\end{remark}

\section{Degenerations of $(3,3)$-polarized Abelian surfaces} \label{sec:w39degenerations}

Now we study the behavior of the variety defined by the $6 \times 6$ Pfaffians of the matrix \eqref{eqn:phimatrix}. The construction starting with the vector only depends on its $G \times \bG_m$-orbit, and hence when restricted to $\fh$, only depends on its orbit under $\rG_{32} \times \bG_m$. So we only need to study one flat from each orbit in Table~\ref{tab:G32flats}. We will show that the description of the Pfaffian locus is uniform for each flat. The summary is as follows:
\begin{compactenum}
\item Family 1 gives smooth Abelian surfaces of degree 18 in $\bP^8$. We have already discussed these in detail in \S\ref{sec:coblecubics}.
\item Families 2 and 3 can be treated together. Here the Pfaffian locus is still a degree $18$ surface. In Family 2, we get the union of the $3$ smooth surfaces, which are ruled surfaces over a common smooth plane cubic (Theorem~\ref{thm:ers}). Family 3 is the degeneration of this case where the smooth plane cubic degenerates to the triangle given by $xyz=0$ (Theorem~\ref{thm:ers-deg}).
\item In Family 4 and lower, we lose flatness. In particular, in Families 4 and 5, the Pfaffian locus jumps up to dimension $3$. Family 4 parametrizes smooth plane cubics $C$, and the Pfaffian locus is the Segre embedding of $C \times \bP^2 \subset \bP^8$. In Family 5, the curve $C$ degenerates to the triangle $xyz=0$. These are treated in \S\ref{sec:planecubics}.
\item In Family 6, the Pfaffian locus is the union of $3$ copies of $\bP^5$. This is the case when the coefficients of the Coble cubic all identically vanish.
\end{compactenum}

\subsection{Family 2: Elliptic ruled surfaces} \label{sec:w39family2}

We choose the flat defined by $c_4=0$ as our representative for Family 2. The general element is 
\[
c_1h_1 + c_2h_2 + c_3h_3.
\]
The condition that this point lies outside of Family 3 is
\begin{align} \label{eqn:f2avoidf3} 
c_1c_2c_3  [(c_1^3+c_2^3-c_3^3)^3 + (3c_1c_2c_3)^3] \ne 0,
\end{align}
and the condition that it lies outside of Family 4 is
\begin{align} \label{eqn:f2avoidf4}
(c_1^3-c_2^3) (c_2^3+c_3^3) (c_1^3+c_3^3) \ne 0. 
\end{align}

On the geometric side, the Jacobian ideal is equal to the quadrics which contain the Pfaffian locus. In fact, the Jacobian locus is the union of the Pfaffian locus with the three $\bP^2$'s defined by
\begin{align*}
z_1 = z_2 = z_4 = z_6 = z_8 = z_9=0, \\ 
z_1 = z_3 = z_5 = z_6 = z_7 = z_8=0, \\
z_2 = z_3 = z_4 = z_5 = z_7 = z_9=0.
\end{align*}

The Pfaffian locus has three components $X_1, X_2, X_3$ which are irreducible for generic choices of $c_1,c_2,c_3$. They are obtained by intersecting $X$ with the $\bP^5$'s defined by $z_1=z_6=z_8=0$, $z_3=z_5=z_7=0$, and $z_2=z_4=z_9=0$, respectively. For a generic choice of $c_1,c_2,c_3$ (i.e., working with the function field in {\tt Macaulay 2}), the minimal generators for the ideal of $X_1$ is given by
\begin{equation} \label{eqn:f2X1eqns}
\begin{split}
z_1, z_6, z_8,\\
c_1(c_2^3+c_3^3)z_4z_5-c_2(c_1^3+c_3^3)z_3z_9-c_3(c_1^3-c_2^3)z_2z_7,\\
c_1(c_2^3+c_3^3)z_7z_9-c_2(c_1^3+c_3^3)z_2z_5-c_3(c_1^3-c_2^3)z_3z_4,\\
c_1(c_2^3+c_3^3)z_2z_3-c_2(c_1^3+c_3^3)z_4z_7-c_3(c_1^3-c_2^3)z_5z_9,\\
c_1c_2c_3(z_3^3+z_5^3+z_7^3) + (c_1^3+c_2^3-c_3^3)z_3z_5z_7,\\
c_1c_2c_3(z_2^3+z_4^3+z_9^3) + (c_1^3+c_2^3-c_3^3)z_2z_4z_9.
\end{split}
\end{equation}
While the quadrics are not $6 \times 6$ Pfaffians, they belong to the radical of the ideal generated by the $6 \times 6$ Pfaffians (by the chain rule applied to taking partial derivatives of $8 \times 8$ Pfaffians), so we may use them in some calculations involving dimension. The ideals of $X_2$ and $X_3$ are similar. From this, we see that the pairwise intersections $X_i \cap X_j$ are (the same) plane cubics written in Hesse normal form (we use variables $x,y,z$):
\begin{align} \label{eqn:family23cubic}
c_1c_2c_3(x^3 + y^3 + z^3) + (c_1^3 + c_2^3 - c_3^3)xyz.
\end{align}

We will see in Proposition~\ref{prop:f2regseq} that each $X_i$ is at most $2$-dimensional. Since $X$ is constructed as a Pfaffian degeneracy locus (scheme-theoretically, not ideal-theoretically), by the generic perfection theorem \cite[Theorem 3.5]{bv}, this is the minimum possible dimension that any of its components can attain, so we conclude that it is exactly $2$-dimensional, and that $\deg(X) = 18$. By symmetry, each of its $3$ components has degree $6$. To prove that \eqref{eqn:f2X1eqns} are in fact all of the equations for $X_1$ whenever \eqref{eqn:f2avoidf3} and \eqref{eqn:f2avoidf4} hold, we need to know that they do not define a surface of degree strictly greater than $6$. It is easy to see that the complete intersection of Proposition~\ref{prop:f2regseq} contains the two $\bP^2$'s defined by $z_3=z_5=z_7=0$ and $z_2=z_4=z_9=0$ as components, and that neither one of them is contained in the surface defined by the equations \eqref{eqn:f2X1eqns}. Hence \eqref{eqn:f2X1eqns} defines a surface of degree at most $6$, and hence is exactly $X_1$.

\begin{proposition} \label{prop:f2regseq}
The $3$ quadrics in the ideal of the varieties $X_i$ (see \eqref{eqn:f2X1eqns}) form a regular sequence.
\end{proposition}

\begin{proof}
We just focus on $X_1$ since the other two are similar. Note that the plane $z_2=z_4=z_9=0$ is contained in the intersection of these $3$ quadrics. If we substitute this into the Jacobian matrix of these $3$ quadrics, we get
\[
\begin{pmatrix}
c_3(c_2^3-c_1^3) z_{7}& c_2(c_1^3+c_3^3) z_{5}& c_1(c_2^3+c_3^3) z_{3}\\
0&      0&      0\\
c_1(c_2^3+c_3^3) z_{5}& c_3(c_1^3-c_2^3) z_{3}& -c_2(c_1^3 +c_3^3) z_{7}\\
0&      0&      0\\
0&      0&      0\\
-c_2(c_1^3+c_3^3) z_{3}&      -c_1(c_2^3+c_3^3) z_{7}& c_3(c_2^3-c_1^3) z_{5}
\end{pmatrix},
\]
so the maximal minors reduce to a single polynomial
\[
(c_2^3+c_3^3) (c_1^3+c_3^3) (c_1^3-c_2^3) (c_1 c_2 c_3 (z_{3}^{3} + z_5^3 + z_7^3) +(c_1^{3}+c_2^{3}-c_3^{3}) z_{3} z_{5} z_{7}).
\]
By the conditions \eqref{eqn:f2avoidf3} and \eqref{eqn:f2avoidf4}, this is a nonzero cubic polynomial, so picking a point $[z_3:z_5:z_7]$ not on the cubic curve gives a point on $X_1$ where the Jacobian has rank $3$, and hence the codimension must be at least $3$ since the rank is bounded above by the codimension of the variety.
\end{proof}

\begin{theorem} \label{thm:ers}
Let $C$ be the plane cubic defined in \eqref{eqn:family23cubic}. There exist line bundles $\cL_1, \cL_2, \cL_3$ of degree $6$ on $C$ so that $X_i = \Proj(\cO_C \oplus \cL_i)$.
\end{theorem}

\begin{proof}
The situation is symmetric, so we focus on $S = X_1$, whose equations are displayed in \eqref{eqn:f2X1eqns}. We will work in the $\bP^5$ defined by $z_1 = z_6 = z_8 = 0$. By Proposition~\ref{prop:f2regseq}, the $3$ quadrics in this list form a complete intersection $T$ of degree $8$. Since $\deg(X_1) = 6$, we see that $X_1$ is linked, in this complete intersection, to a degree $2$ variety, which is in fact the union of two $\bP^2$'s which are defined by $z_3 = z_5 = z_7 = 0$ and $z_2 = z_4 = z_9 = 0$, respectively. We will call them $\bP$ and $\bP'$.

Let $\cI_S \subset \cO_T$ be the ideal sheaf of $S$, and let $\cI_P \subset \cO_T$ be the ideal sheaf of $P = \bP \cup \bP'$. From the Koszul complex, we see that $\cO_T$ is the dualizing sheaf of $T$. Then we have 
\[
\cI_S = \cH om_{\cO_T}(\cO_{P}, \cO_T) = \omega_{\bP} \oplus \omega_{\bP'} = \cO_{\bP}(-3) \oplus \cO_{\bP'}(-3),
\]
where the first equality follows from \cite[Theorem 21.23]{eisenbud} and the second equality follows from \cite[Theorem 21.15]{eisenbud}. Since $\bP \cap \bP' = \emptyset$, their union is a Cohen--Macaulay scheme, so by \cite[Theorem 21.23(b)]{eisenbud}, $S$ is also Cohen--Macaulay. Similarly, we see that $\cI_P = \omega_S$. The image of $\cI_P$ in $\cO_S$ is generated by $2$ cubics, but these cubic curves are disjoint, so it is locally generated by a single cubic. In particular, by \cite[Theorem 21.23(c)]{eisenbud}, $S$ is Gorenstein and $\omega_S$ is a line bundle.

In particular, we have the following Koszul complexes
\begin{align*}
0 \to \cO_{\bP^5}(-6)^2 \to \cO_{\bP^5}(-5)^6 \to \cO_{\bP^5}(-4)^6 \to \cO_{\bP^5}(-3)^2 \to \cI_S \to 0,\\
0 \to \cO_{\bP^5}(-6) \to \cO_{\bP^5}(-4)^3 \to \cO_{\bP^5}(-2)^3 \to \cO_{\bP^5} \to \cO_T \to 0.
\end{align*}
From the short exact sequence $0 \to \cI_S \to \cO_T \to \cO_S \to 0$, we construct a mapping cone of the two complexes to get a locally free resolution of length $4$ for $\cO_S$. We can replace the last differential of the mapping cone $\cO_{\bP^5}(-6)^2 \to \cO_{\bP^5}(-6) \oplus \cO_{\bP^5}(-5)^6$ with its cokernel (the cokernel is locally free by the Auslander--Buchsbaum formula \cite[Theorem 19.9]{eisenbud} since $S$ is Cohen--Macaulay). The map $\cO_{\bP^5}(-6)^2 \to \cO_{\bP^5}(-6)$ must be nonzero (otherwise by Ext duality, the map $\cI_S \to \cO_T$ is zero). In fact, the map is nonzero on both factors of $\cO_{\bP^5}(-6)$ in the target, so we remove a diagonal copy of $\cO_{\bP^5}(-6)$ from the source which maps isomorphically to the target, and are left with the cokernel of $\cO_{\bP^5}(-6) \to \cO_{\bP^5}(-5)^6$. Since this map is built out of two copies of the Koszul complex for the ideals $(z_2,z_4,z_9)$ and $(z_3,z_5,z_7)$, we see that after a suitable choice of bases, this map is the column vector $(z_2,z_3,z_4,z_5,z_7,z_9)^T$, so we can identify the cokernel with $\cQ(-5)$, where $\cQ$ is the tautological quotient bundle on $\bP^5$. In particular, we get the following locally free resolution
\[
0 \to \cQ(-5) \to \cO_{\bP^5}(-4)^9 \to \cO_{\bP^5}(-3)^2 \oplus \cO_{\bP^5}(-2)^3 \to \cO_{\bP^5} \to \cO_S \to 0.
\]
Now we apply $\cH om_{\bP^5}(-,\cO_{\bP^5}(-6))$ to get the locally free resolution
\begin{align} \label{eqn:linkagelfr}
0 \to \cO_{\bP^5}(-6) \to \cO_{\bP^5}(-4)^3 \oplus \cO_{\bP^5}(-3)^2 \to \cO_{\bP^5}(-2)^9 \to \Omega_{\bP^5} \to \omega_S \to 0.
\end{align}
Since $S$ is Gorenstein, $\omega_S$ is a line bundle, and so the kernel of the surjection $\Omega_{\bP^5}(1) \otimes \cO_S \to \omega_S(1)$ is a rank $4$ vector bundle $\cK$. Using the Euler exact sequence
\[
0 \to \Omega_{\bP^5}(1) \to \cO_{\bP^5}^6 \to \cO_{\bP^5}(1) \to 0
\]
restricted to $S$, we get a partial flag $\cK \subset \Omega_{\bP^5}(1)|_S \subset \cO_S^6$ whose successive quotients are line bundles, and hence a morphism
\[
\sigma \colon S \to {\bf Fl}(4,5;6) = F,
\]
where $F$ is the variety of partial flags of subspaces with dimensions $4$ and $5$. If $p_2 \colon F \to \bP^5$ is the second projection, then $p_2 \sigma$ maps $S$ to itself, and hence $\sigma$ is an embedding. Let $p_1 \colon S \to \Gr(4,6)$ be the restriction of the first projection map.

We can identify $F$ with the associated projective bundle of $\cQ^*(1)$ over $\bP^5$. Take the map $\cO_{\bP^5}(-2)^9 \to \Omega_{\bP^5}$ from \eqref{eqn:linkagelfr}, twist by $\cO_{\bP^5}(2)$ and pullback to $F$. Then $\sigma(S)$ is the locus where this map fails to be surjective. Note that $\rH^0(F; p_2^*(\Omega_{\bP^5}(2)))$ is naturally identified with $\rH^0(\Gr(4,6); \cO_{\Gr(4,6)}(1))$. So taking sections of the map $\cO_F^9 \to p_2^*(\Omega_{\bP^5}(2))$, the image gives $9$ linear relations in the ambient space of the Pl\"ucker embedding of $\Gr(4,6)$, and from the surjectivity statement above, we see that $\sigma(S) = p_1^{-1}(\Gamma)$ where $\Gamma$ is the zero locus of these $9$ equations in $\Gr(4,6)$. In particular, the fibers of $p_1$ are all projective lines, so $\Gamma$ is a curve, and $S$ is a ruled surface over $\Gamma$.

Earlier, we saw that $\cI_P = \omega_S$. In particular, the ideal sheaf of $(\bP \cup \bP') \cap S$ in $\cO_S$ is $\omega_S$. The pullback via $p_1$ of the ample line bundle on $\Gr(4,6)$ to $S$ is $\omega_S(2)$, so identifying $S$ with $\sigma(S) \subset F$, we get 
\[
\omega_S = \cO_S \otimes p_1^*(\cO_{\Gr(4,6)}(1)) \otimes p_2^*(\cO_{\bP^5}(-2)).
\]
We know that $(\bP \cup \bP') \cap S$ consists of two disjoint curves $C$, $C'$, each contained in a disjoint copy of $\bP^2$. We claim that the restriction maps $p_1 \colon C \to \Gamma$ and $p_1 \colon C' \to \Gamma$ are isomorphisms. Suppose that the first map is not an isomorphism. Then $C$ intersects some fiber of $p_1 \colon F \to \Gr(4,6)$ in at least $2$ points. But this fiber represents a line in $S$, and $C$ is a plane cubic curve in $\bP^2$. This means that the line must lie in this $\bP^2$ and hence intersects $C$ in at least $3$ points, which contradicts that the ideal of $C \cup C'$ has bidegree $(1,2)$.

In conclusion, $S$ is a ruled surface over $\Gamma$, which is a genus $1$ curve. In particular, $S$ is the associated projective bundle of a rank $2$ bundle over $\Gamma$. Since we have two disjoint sections of the ruling, this rank $2$ bundle splits into a direct sum of $2$ line bundles $L \oplus L'$. The bundle is well-defined up to twisting by a line bundle, so we may normalize the line bundles to look like $\cO_\Gamma \oplus \cL_1$.
\end{proof}

\begin{remark}
Since Family 2 gives a $2$-dimensional parameter space, the triple $(\cL_1, \cL_2, \cL_3)$ from Theorem~\ref{thm:ers} is not general. We expect these triples of line bundles to be related to the triples described in \cite[\S 2.2]{weiho}.
\end{remark}

\begin{proposition} \label{prop:f2cubicsmooth}
The plane cubic \eqref{eqn:family23cubic} is smooth. In particular, each $X_i$ is smooth.
\end{proposition}

\begin{proof}
We recall that the cubic curve $\lambda(x^3+y^3+z^3) + \mu xyz = 0$ is smooth if and only if $\lambda(\mu^3 + 27\lambda^3) \ne 0$. This is equivalent to \eqref{eqn:f2avoidf4} by setting $\lambda = c_1c_2c_3$ and $\mu = c_1^3 + c_2^3 - c_3^3$. The last statement follows since each $X_i$ is a ruled surface over the plane cubic \eqref{eqn:family23cubic} by Theorem~\ref{thm:ers}.
\end{proof}

\subsection{Family 3: Degenerate elliptic ruled surfaces} \label{ss:w39family3}

We consider the flat defined by $c_3 = c_4 = 0$. Hence we may parametrize this as the family
\[
c_1 h_1 + c_2 h_2.
\]
This defines a line in the Burkhardt quartic with equations $\gamma_1 = \gamma_4 = \gamma_5 = 0$. The condition that this point lies outside Family 5 is that $c_1^6 - c_2^6 \ne 0$, and the condition that the point lies outside Family 6 is that $c_1c_2 \ne 0$. Hence we will assume that
\begin{align} \label{eqn:family3disc}
c_1c_2(c_1^6 - c_2^6) \ne 0.
\end{align}

The Jacobian ideal \eqref{eqn:jacobian} is the intersection of the Pfaffian ideal with $6$ linear ideals:
\begin{align*}
\begin{array}{ll}
      z_{1} = z_{2} = z_{4} = z_{6} = z_{8} = z_{9} = 0,&
      z_{1} = z_{2} = z_{5} = z_{6} = z_{7} = z_{9} = 0,\\
      z_{1} = z_{3} = z_{4} = z_{5} = z_{8} = z_{9} = 0,&
      z_{1} = z_{3} = z_{5} = z_{6} = z_{7} = z_{8} = 0,\\
      z_{2} = z_{3} = z_{4} = z_{5} = z_{7} = z_{9} = 0,&
      z_{2} = z_{3} = z_{4} = z_{6} = z_{7} = z_{8} = 0.
\end{array}
\end{align*}

Using that $c_2(c_1^6 - c_2^6) \ne 0$, the Pfaffian ideal is minimally generated by $9$ quadrics and $7$ cubics. This is a degeneration of Family 2, so we can decompose the Pfaffian ideal into $3$ pieces $X_1, X_2, X_3$. However, they are no longer irreducible.
In fact, each one decomposes further into $3$ more pieces, and we write $X_i = X_{i,1} \cup X_{i,2} \cup X_{i,3}$. Each one lies in a $\bP^3$ and the equations are of the form $c_2^2 xy - c_1^2 zw$. For example, $X_1$ is the union of
\begin{align*}
z_8 = z_7 = z_6 = z_2 = z_1 = c_2^2 z_4z_5 -  c_1^2 z_3z_9 = 0,\\
z_9 = z_8 = z_6 = z_5 = z_1 = c_2^2 z_2z_3 - c_1^2 z_4z_7 = 0,\\
z_8 = z_6 = z_4 = z_3 = z_1 = c_2^2 z_7z_9 - c_1^2 z_2z_5 = 0.
\end{align*}

The three quadrics containing the union of these three varieties are
\begin{align*}
c_1^2z_3z_9-c_2^2 z_4z_5,\qquad 
c_1^2 z_2z_5-c_2^2z_7z_9,\qquad 
c_1^2z_4z_7-c_2^2 z_2z_3.
\end{align*}

\begin{proposition}
These three quadrics form a regular sequence.
\end{proposition}

\begin{proof}
If we substitute $z_2=z_7=0$, then the equations for the quadrics reduces to $q = c_1^2 z_3z_9 - c_2^2z_4z_5 = 0$. Also, if we substitute $z_2=z_7$ into the Jacobian of these three quadrics and take the $3 \times 3$ minors, we get the equations
\begin{align*}
\begin{array}{llll}
z_9(c_1^4z_4z_5-c_2^4z_3z_9),&
z_5(c_1^4z_4z_5-c_2^4z_3z_9),&
z_4(c_1^4z_4z_5-c_2^4z_3z_9),&
z_3(c_1^4z_4z_5-c_2^4z_3z_9).
\end{array}
\end{align*}
If $q$ divides any of these minors, then we have $c_1^6 = c_2^6$, which violates \eqref{eqn:family3disc}. Hence we may choose a point on $q=0$ for which at least one of these minors is nonzero. This implies that the intersection of these three quadrics contains a point where the Jacobian has rank 3, and hence the quadrics cut out a codimension $3$ variety, since the rank of the Jacobian is bounded from above by the codimension.
\end{proof}

\begin{theorem} \label{thm:ers-deg}
Each $X_i$ is a ruled surface over the cubic curve $xyz = 0$.
\end{theorem}

\begin{proof}
The proof is the same as the proof of Theorem~\ref{thm:ers}.
\end{proof}

\begin{remark}
Any two $X_i$, $X_j$ either intersect in a linear $\bP^1$ or in a point. For a fixed component, each case happens $4$ times. Hence we can form a $4$-regular graph on $9$ vertices. Let $\Gamma$ be its automorphism group. Then $\# \Gamma = 72$ and there is a split short exact sequence
\[
1 \to \bZ/3 \times \bZ/3 \to \Gamma \to {\rm Di}_4 \to 1
\]
where ${\rm Di}_4$ is the dihedral group of the square. The wreath product $G(6,1,2) = S_2 \ltimes (\bZ/6)^2$ also has this property, but the groups are not isomorphic. According to {\tt GAP} small groups notation \cite{gap}, $\Gamma$ is group number $(72,40)$ and $G(6,1,2)$ is group number $(72, 30)$. 

To construct this semidirect product, note that $\bZ/3 \times \bZ/3$ is the group of translations of $\bA^2_{\bF_3}$ and that ${\rm Di}_4$ is the stabilizer of two points in $\bP^1_{\bF_3}$. More precisely, the data of the graph can be reinterpreted by picking two axes of direction in $\bA^2_{\bF_3}$, which correspond to two points in $\bP^1_{\bF_3}$.
\end{remark}

\subsection{Families 4 and 5: Plane cubics} \label{sec:planecubics}

We first consider Family 4. This is the flat defined by $c_4 = c_2 + c_3 = 0$. Hence we may parametrize this as the family 
\[
c_1 h_1 + c_2(h_2 - h_3).
\]
This defines the point $[0:0:1:-1:0]$ on the Burkhardt quartic $\cB$, and is one of its singular points. We remark that the exceptional divisor of the blowup of $\cB$ at this point (and every singular point) has an interpretation as the moduli of products of plane cubics \cite[Theorem 3.1]{geer}. The condition that this point lies outside Family 5 is $c_1(c_1^3 + 8c_2^3) \ne 0$ and the condition that this point lies outside Family 6 is $c_2(c_1^3 - c_2^3) \ne 0$. Hence we will assume that
\begin{align} \label{eqn:f4disc}
c_1c_2(c_1^3-c_2^3)(c_1^3+8c_2^3) \ne 0.
\end{align}

The Jacobian ideal \eqref{eqn:jacobian} is generated by the $9$ quadrics which are $c_2(c_1^3 - c_2^3)$ times the $2 \times 2$ minors of the matrix 
\begin{align} \label{eqn:f4segre}
\begin{pmatrix} 
z_1 & z_2 & z_3 \\ z_6 & z_4 & z_5 \\ z_8 & z_9 & z_7 
\end{pmatrix}.
\end{align}
So its zero locus is $\bP^2 \times \bP^2$ and has dimension $4$.

The Pfaffian locus has dimension $3$ and degree $9$. It is defined by the $9$ quadrics of the Jacobian ideal above and $10$ additional cubics 
\begin{align*} 
\begin{array}{ll}
c_1c_2^2(z_1^3 + z_6^3 + z_8^3)-(c_1^3+2c_2^3)z_1z_6z_8, &
c_1c_2^2(z_2^3 + z_4^3 + z_9^3)-(c_1^3+2c_2^3)z_2z_4z_9, \\ 
c_1c_2^2(z_3^3 + z_5^3 + z_7^3) - (c_1^3+2c_2^3)z_3z_5z_7,&
c_1c_2^2(z_1z_2z_3 + z_4z_5z_6 + z_7z_8z_9)-(c_1^3+2c_2^3)z_3z_6z_9,\\ 
c_1c_2^2(z_2z_3^2 + z_4z_5^2 + z_9z_7^2) -(c_1^3+2c_2^3)z_3z_5z_9,&
c_1c_2^2(z_1z_3^2 + z_6z_5^2 + z_8z_7^2) -(c_1^3+2c_2^3)z_3z_5z_8,\\
c_1c_2^2(z_3z_2^2 + z_5z_4^2 + z_7z_9^2)-(c_1^3+2c_2^3)z_2z_5z_9,& 
c_1c_2^2(z_1z_2^2 + z_6z_4^2 + z_8z_9^2)-(c_1^3+2c_2^3)z_2z_6z_9,\\ 
c_1c_2^2(z_3z_1^2 + z_5z_6^2 + z_7z_8^2)-(c_1^3+2c_2^3)z_3z_6z_8,& 
c_1c_2^2(z_2z_1^2 + z_4z_6^2 + z_9z_8^2)-(c_1^3+2c_2^3)z_1z_6z_9.
\end{array}
\end{align*}

Let $V$, $W$ be two $3$-dimensional vector spaces with coordinate functions $v_1,v_2,v_3$ and $w_1,w_2,w_3$, and identify the matrix in \eqref{eqn:f4segre} with $\begin{pmatrix} v_1 & v_2 & v_3 \end{pmatrix}^T \begin{pmatrix} w_1 & w_2 & w_3 \end{pmatrix}$. Let $C$ be the curve 
\[
c_1c_2^2(v_1^3 + v_2^3 + v_3^3) - (c_1^3 + 2c_2^3)v_1v_2v_3 = 0
\]
in $\bP(V)$. Then our variety above is the Segre embedding of $C \times \bP(W)$. 

\begin{proposition} \label{prop:f4cubicsmooth}
$C$ is smooth if \eqref{eqn:f4disc} holds.
\end{proposition}

\begin{proof}
This is exactly the same as Proposition~\ref{prop:f2cubicsmooth}.  
\end{proof}

Family 5 is the limiting case of Family 4 where $c_1 \to 0$. In particular, we consider the vector $c_2(h_2-h_3)$ where $c_2 \ne 0$. The curve $C$ has degenerated to the triangle $xyz=0$, so the Pfaffian locus now has $3$ components.

\section{Genus 2 analogue of Shioda's surface} \label{sec:shioda}

\subsection{The Coble--Shioda variety}

Define $\CS^\circ$ as the union of all smooth Abelian surfaces arising from the construction in \S\ref{ss:coblecubics}, i.e., the set of all points $\bz$ that satisfy the equations \eqref{eqn:jacobian} for some $(c_1, \dots, c_4)$ with $\Delta(\bc) \ne 0$, as defined in \eqref{eqn:discw39}. We denote the closure by $\CS = \ol{\CS^\circ}$, and call it the {\bf Coble--Shioda variety}. 

Our goal is to give explicit determinantal equations for $\CS$, which can be found in Theorem~\ref{thm:cobleshioda}. We will see that these determinantal equations do not generate a prime ideal. In particular, we present $4$ extra equations and conjecture that they suffice to get the whole ideal (Conjecture~\ref{conj:CSradical}). In \S\ref{ss:CSminors}, we analyze the zero sets of the lower order minors from Theorem~\ref{thm:cobleshioda}.

In order to study the variety $\CS$, it is useful to have alternative descriptions. Define $\CS_{\Jac}$ to be the closure of the union of the loci defined by the Jacobian ideal \eqref{eqn:jacobian} as we vary over all $\bc \in \bP^3 \setminus \text{Family 6}$. Also define $\CS_{\Pf}$ to be the closure of the union of the loci defined by the $6 \times 6$ Pfaffians of the matrix $\Phi_\bc(z)$ from \eqref{eqn:phimatrix} as we vary over all $\bc \in \bP^3 \setminus \text{Family 6}$.

\begin{theorem} \label{thm:CSdefn}
$\CS = \CS_{\Pf} = \CS_{\Jac}$.
\end{theorem}

The proof of this theorem will follow from the next three results.

\begin{proposition} \label{prop:CSirred}
$\CS$ is an irreducible variety.
\end{proposition}

\begin{proof}
It is enough to show that $\CS^\circ$ is irreducible. Consider the set
\[
Z^\circ = \{(\bc,z) \in \bP^3 \times \bP^8 \mid \rank \Phi_\bc(z) < 6,\ \Delta(\bc) \ne 0\}.
\]
The first projection $\pi_1$ gives a map from $Z^\circ$ to the complement of $\Delta(\bc) = 0$, which is irreducible. The fibers of $\pi_1$ are smooth Abelian surfaces, and in particular are irreducible and of constant dimension $2$. We conclude that $Z^\circ$ is irreducible \cite[Exercise 14.3]{eisenbud}. Since $\pi_2(Z^\circ) = \CS^\circ$, we get that $\CS^\circ$ is irreducible. 
\end{proof}

\begin{proposition} \label{prop:CS=CSpf}
$\CS = \CS_{\Pf}$.
\end{proposition}

\begin{proof}
Define
\[
Z = \{(\bc,z) \in (\bP^3 \setminus \text{Family 6}) \times \bP^8 \mid \rank \Phi_\bc(z) < 6\}.
\]
If we view $(\bP^3 \setminus \text{Family 6}) \times \bP^8$ as a relative $\bP^8$ over the base $\bP^3 \setminus \text{Family 6}$, then we have a relative hyperplane bundle $\cQ$ and $Z$ is the Pfaffian degeneracy locus of a section of $\bigwedge^2 \cQ \otimes \cO(1)$, and hence by the generic perfection theorem \cite[Theorem 3.5]{bv} (applied to the coordinate ring of the ideal of $6 \times 6$ Pfaffians of a generic $8 \times 8$ skew-symmetric matrix), each of its irreducible components has codimension at most $6$, i.e., dimension at least $5$. By our analysis in \S\ref{sec:w39degenerations}, we see that the preimages of Family 2, 3, 4, and 5 have dimensions $4$, $3$, $4$, and $3$, respectively (we are just adding the dimension of the family to the dimension of the fiber, which is constant), so they cannot be irreducible components. 
This means that $Z$ is the closure of $Z^\circ$ from the proof of Proposition~\ref{prop:CSirred}. We have already seen that $Z^\circ$ is irreducible, so $Z$ is irreducible, and since $\CS_{\Pf}$ is the closure of $\pi_2(Z)$, it is also irreducible. Finally, $\CS$ is irreducible by Proposition~\ref{prop:CSirred}, and we have $\CS \subseteq \CS_{\Pf}$ and both of them have the same dimension, so they are equal.
\end{proof}

\begin{proposition} \label{prop:CS=CSJac}
$\CS_{\Pf} = \CS_{\Jac}$.
\end{proposition}

\begin{proof}
We just need to analyze the discrepancy between the zero sets defined by the Pfaffian equations $\rank \Phi_\bc(z) < 6$ from \eqref{eqn:phimatrix} and Jacobian equations \eqref{eqn:jacobian} for each family in Table~\ref{tab:G32flats}. For the generic family they define the same zero set.

For Families 2 and 3, we saw in \S\ref{sec:w39family2} that the discrepancy is given by a union of $3$ linear $\bP^2$'s. The Pfaffian locus contains a plane cubic in each of these planes, and in total, these plane cubics form the Hesse pencil and fill out the planes. So we see that in this case, the union over all Pfaffian loci in Families 2 and 3 coincides with the union over all the Jacobian loci in Families 2 and 3.

For Families 4 and 5, we saw in \S\ref{sec:planecubics} that the Jacobian locus is a Segre product $\bP^2 \times \bP^2$ and the Pfaffian locus is of the form $C \times \bP^2$ for a plane cubic curve $C$ in Hesse normal form. So again, we see that the Pfaffian loci and Jacobian loci agree after taking the union over all elements in Families 4 and 5.
\end{proof}

\subsection{Determinantal equations}

\begin{theorem} \label{thm:cobleshioda}
The Coble--Shioda variety is defined set-theoretically by the maximal minors of the matrix
\[
\rCS(z) = \begin{pmatrix}
z_{1}^{2}& z_{2}^{2}& z_{3}^{2}& z_{4}^{2}& z_{5}^{2}& z_{6}^{2}& z_{7}^{2}&    z_{8}^{2}& z_{9}^{2}\\
z_{2} z_{3} & z_{1} z_{3} & z_{1} z_{2} & z_{5} z_{6} & z_{4} z_{6} & z_{4} z_{5} & z_{8} z_{9} & z_{7} z_{9} & z_{7} z_{8}\\
z_{4} z_{7} & z_{5} z_{8} & z_{6} z_{9} & z_{1} z_{7} & z_{2} z_{8} & z_{3} z_{9} & z_{1} z_{4} & z_{2} z_{5} & z_{3} z_{6}\\
z_{5} z_{9} & z_{6} z_{7} & z_{4} z_{8} & z_{3} z_{8} & z_{1} z_{9} & z_{2} z_{7} & z_{2} z_{6} & z_{3} z_{4} & z_{1} z_{5}\\
z_{6} z_{8} & z_{4} z_{9} & z_{5} z_{7} & z_{2} z_{9} & z_{3} z_{7} & z_{1} z_{8} & z_{3} z_{5} & z_{1} z_{6} & z_{2} z_{4}
\end{pmatrix}
\]
\end{theorem}

\begin{remark}
We can interpret $\rCS(z)$ as the map (tensor)
\[
\rH^0(\bP^8; \cO_{\bP^8}(3))^{H_{3,2}} \otimes \rH^0(\bP^8; \cO_{\bP^8}(1))^* \to \rH^0(\bP^8; \cO_{\bP^8}(2)), 
\]
given by taking partial derivative, where $H_{3,2}$ is the Heisenberg group from \S\ref{sec:heisenberg}.
\end{remark}

\begin{proof}
If $[z_1 : \cdots : z_9]$ is in the zero locus of the Jacobian ideal for some $\gamma$, then $\rCS(z)$ does not have full rank, so the union of all zero loci of Jacobian ideals is a subset of the variety of maximal minors of $\rCS(z)$. Conversely, suppose that for $[z_1 : \cdots : z_9]$, we know that $\rCS(z)$ does not have full rank. By the discussion in \S\ref{ss:CSminors} below, the rank cannot be $3$. If it has rank $\le 2$, then it lies on the union of 120 linear $\bP^2$'s. These $\bP^2$'s are subsets of the varieties defined by the Jacobian ideals in Family 3, as follows from the discussion in \S\ref{ss:w39family3}. So by Theorem~\ref{thm:CSdefn}, we have $z \in \CS$.

So we may suppose that $\rank \rCS(z) = 4$. As in the proof of Theorem~\ref{thm:shiodasurface}, for each choice of $4$ columns $i,j,k,\ell$, we can find a linear dependence $v(i,j,k,\ell)$ among the rows. We claim that the entries of each $v(i,j,k,\ell)$ satisfy the relation \eqref{eqn:burkhardt} modulo the ideal of maximal minors. In fact, it is enough to check the subsets $\{1,2,3,4\}$ and $\{1,2,4,5\}$ since the group of affine transformations $H = \GL_2(\bF_3) \ltimes \bF_3^2$ acts as a symmetry group for these subsets, and there are only two configurations up to symmetry: either the $3$ of the $4$ points lie on an affine line in $\bA_{\bF_3}^2$ (and hence is equivalent to $\{1,2,3,4\}$) or they don't (and hence is equivalent to $\{1,2,4,5\}$. For these two subsets, we use {\tt Macaulay 2} to check directly.

So the maximal minors define a point on the Burkhardt quartic $\cB$. By Proposition~\ref{prop:cmapsurj}, this point will be in the image of the map \eqref{eqn:cmap} unless it is one of the special $160$ points. However, we can see directly that for any such point, the matrix $\rCS(z)$ does not have rank $4$. For example, taking the point $[0:0:0:0:1]$, the Jacobian ideal is generated by the partial derivatives of the cubic equation $z_1z_6z_8 + z_2z_4z_9 + z_3z_5z_7 = 0$. This is the ideal of the union of $27$ linearly embedded $\bP^2$'s whose equations are defined by the vanishing of $6$ variables, where we have chosen $2$ from each of the sets $\{z_1,z_6,z_8\}$, $\{z_2,z_4,z_9\}$, $\{z_3,z_5,z_7\}$. These are some of the $120$ linear $\bP^2$'s in the rank $\le 2$ locus discussed above. Hence we get $z \in \CS$.
\end{proof}

The proof of Theorem~\ref{thm:cobleshioda} gives a characterization of singular Heisenberg-invariant cubics.

\begin{corollary}
The closure of the locus of Coble cubics in the space of all Heisenberg-invariant cubics on $\bP^8$ is the locus of singular Heisenberg-invariant cubics.
\end{corollary}

The following $4$ equations belong to the radical of the ideal of maximal minors of $\rCS(z)$ (in fact the square of each equation is in the determinantal ideal):
\small\begin{align*}
z_2z_6z_7(z_1^3-z_3^3-z_4^3+z_5^3-z_8^3+z_9^3) + z_3z_4z_8(-z_1^3+z_2^3-z_5^3+z_6^3+z_7^3-z_9^3) + z_1z_5z_9(-z_2^3 + z_3^3 + z_4^3 - z_6^3 - z_7^3 + z_8^3),\\ 
z_3z_5z_7(z_1^3 - z_2^3 - z_4^3 + z_6^3 + z_8^3 - z_9^3) + z_1z_6z_8(z_2^3 - z_3^3 + z_4^3 - z_5^3 - z_7^3 + z_9^3) + z_2z_4z_9(-z_1^3 + z_3^3 + z_5^3 - z_6^3 + z_7^3 - z_8^3),\\
z_1z_4z_7(z_2^3 - z_3^3 + z_5^3 - z_6^3 + z_8^3 - z_9^3) + z_2z_5z_8(-z_1^3 + z_3^3 - z_4^3 + z_6^3 - z_7^3 + z_9^3) + z_3z_6z_9(z_1^3 - z_2^3 + z_4^3 - z_5^3 + z_7^3 - z_8^3),\\
z_1z_2z_3(z_4^3 + z_5^3 + z_6^3 - z_7^3 - z_8^3 - z_9^3) + z_4z_5z_6(-z_1^3 - z_2^3 - z_3^3 + z_7^3 + z_8^3 + z_9^3) + z_7z_8z_9(z_1^3 + z_2^3 + z_3^3 - z_4^3 - z_5^3 - z_6^3).
\end{align*}
\normalsize
These were found as follows. The variety $\{(\gamma,z) \in \cB \times \bP^8 \mid I_\gamma(z) = 0\}$, where $I_\gamma$ is the set of equations in \eqref{eqn:jacobian}, has some obvious equations, namely, the quartic equation defining $\cB \subset \bP^4$, and the $9$ bidegree $(1,2)$ equations from \eqref{eqn:jacobian} (with the quartic expressions in $c_i$ replaced by the coordinates on $\cB$). Let $I$ be this ideal and let $y_1, \dots, y_5$ be the coordinates on $\bP^4$. Then we executed the following command in {\tt Macaulay 2}
\begin{verbatim}
eliminate({y_1..y_5}, quotient(I,ideal(y_1..y_5)))
\end{verbatim}
to find these $4$ equations.
If we add these equations, we get a saturated ideal with Hilbert series
\begin{align*}
\frac{1+3T+6T^2+10T^3+15T^4+21T^5+24T^6+24T^7+21T^8}{(1-T)^6} + \qquad \\*
\qquad \qquad \qquad \frac{15T^9-30T^{10}-60T^{11}+105T^{12}-75T^{13}+30T^{14}-5T^{15}}{(1-T)^{6}},
\end{align*}
and this Betti table
\small \begin{Verbatim}[samepage=true]
       0  1   2   3   4   5   6  7 8
total: 1 40 234 540 630 425 180 45 5
    0: 1  .   .   .   .   .   .  . .
       ...
    5: .  4   .   .   .   .   .  . .
       ...
    9: . 36  54   .   .   .   .  . .
   10: .  . 180 540 630 425 180 45 5
\end{Verbatim}
\normalsize

\begin{conjecture} \label{conj:CSradical}
The maximal minors of $\rCS(z)$ and the $4$ additional equations above generate a prime ideal.
\end{conjecture}

\begin{remark}
The graded Betti table for the cokernel of $\rCS(z)$ is self-dual:
\small \begin{Verbatim}[samepage=true]
       0 1 2 3
total: 5 9 9 5
    0: 5 . . .
    1: . 9 . .
       ...
    5: . . 9 .
    6: . . . 5
\end{Verbatim}
\normalsize

The middle differential is a skew-symmetric matrix after a suitable choice of bases. A typical entry in the matrix looks like
$z_1z_3(z_4^3 + z_5^3 + z_6^3 - z_7^3 - z_8^3 - z_9^3) + 3z_2^2(-z_4z_5z_6 + z_7z_8z_9)$.
The matrix has rank $4$, and the $4 \times 4$ Pfaffians define the same ideal as the maximal minors of $\rCS(z)$.
\end{remark}

\begin{remark} \label{rmk:w39section}
$\rCS(z)$ generically has rank $4$ on $\bP_\rM$, so we see that $\bP_\rM$ is contained in the Coble--Shioda variety. Its rank drops to $2$ on $40$ points, which will correspond to the base locus of the degree $6$ map $\bP^3 \dashrightarrow \cB$. The matrix never has rank below $2$ on $\bP_\rM$. By taking minors of $\rCS(z)$ over one of the 6 points of intersection $\bP_\rM \cap X$, for some Abelian surface $X$, we recover its Burkhardt coordinates, so we get a degree $6$ rational map $\bP_\rM \dashrightarrow \cB$. But we can do better. By \S\ref{ss:coblecubics}, we see that each of the $6$ intersection points will come from a $9 \times 9$ skew-symmetric matrix. Ultimately, we get a section $\bP^3 \to \bP^8$ defined by 
\[
[c_1:c_2: c_3:c_4] \mapsto [0: -c_1 : c_1: -c_2: -c_3: -c_4: c_2: c_4: c_3]
\]
and this maps isomorphically to $\bP_\rM$. We can get $80$ other sections by taking translates by the Heisenberg group.
\end{remark}

\subsection{Lower-order minors.} \label{ss:CSminors}

The variety defined by the $4 \times 4$ minors of $\rCS(z)$ coincides with the variety defined by the $3 \times 3$ minors of $\rCS(z)$ and is a surface of degree $120$. We claim that it is a union of $120$ $\bP^2$'s.

\begin{compactenum}[(1)]
\item For every line $\{i,j,k\}$ in $\bA^2_{\bF_3}$, we have the $\bP^2$ defined by $z_\ell = 0$ for $\ell \notin \{i,j,k\}$. This gives $12$ $\bP^2$'s.

\item For every point in $\bP^1_{\bF_3}$, we can define $27$ different $\bP^2$'s in the $3 \times 3$ minors. Let $\omega$ be a fixed cube root of unity. Then, for example, 
\begin{align*}
z_1 = \omega^{i_1} z_2 = \omega^{j_1} z_3, \qquad 
z_4 = \omega^{i_2} z_5 = \omega^{j_2} z_6, \qquad 
z_7 = \omega^{i_3} z_8 = \omega^{j_3} z_9.
\end{align*}
defines such a $\bP^2$ if and only if $i_1 + i_2 + i_3 = j_1 + j_2 + j_3 = i_1 - i_2 + j_1 - j_2 = 0$ modulo $3$. The other $3$ points in $\bP^1_{\bF_3}$ are handled in a similar way, and they are:
\begin{align*}
z_1 = \omega^{i_1} z_4 = \omega^{j_1} z_7, \qquad 
z_2 = \omega^{i_2} z_5 = \omega^{j_2} z_8, \qquad 
z_3 = \omega^{i_3} z_6 = \omega^{j_3} z_9;\\
z_1 = \omega^{i_1} z_6 = \omega^{j_1} z_8, \qquad 
z_2 = \omega^{i_2} z_4 = \omega^{j_2} z_9, \qquad 
z_3 = \omega^{i_3} z_5 = \omega^{j_3} z_7;\\
z_1 = \omega^{i_1} z_5 = \omega^{j_1} z_9, \qquad 
z_2 = \omega^{i_2} z_6 = \omega^{j_2} z_7, \qquad 
z_3 = \omega^{i_3} z_4 = \omega^{j_3} z_8.
\end{align*}
\end{compactenum}

This gives $4 \times 27 = 108$ $\bP^2$'s. For any two $\bP^2$'s, they are either disjoint, or intersect in exactly 1 point. In particular, this variety is not Cohen--Macaulay.

If we intersect the ideals of all of these hyperplanes, we get the saturation of the ideal of $3 \times 3$ minors. It is generated by $36$ quintics and $480$ sextics. Here is its Betti table
\small \begin{Verbatim}[samepage=true]
       0   1    2    3    4    5    6   7 8
total: 1 516 2736 5715 6221 3995 1575 316 9
    0: 1   .    .    .    .    .    .   . .
    1: .   .    .    .    .    .    .   . .
    2: .   .    .    .    .    .    .   . .
    3: .   .    .    .    .    .    .   . .
    4: .  36    .    .    .    .    .   . .
    5: . 480 2736 5634 5816 3159  891 100 .
    6: .   .    .    .    .    .    .   . .
    7: .   .    .   81  405  836  684 216 .
    8: .   .    .    .    .    .    .   . 9
\end{Verbatim}
\normalsize

Its Hilbert series is
\[
\frac{1+6T+21T^2+56T^3+126T^4+216T^5-234T^6-108T^7-27T^8+54T^9+9T^{10}}{(1-T)^3}
\]

The variety of $2 \times 2$ minors consists of $360$ reduced points.

\begin{compactenum}[(1)]
\item There are $9$ singular points for the permutations of $[1: 0: 0: 0: 0: 0: 0: 0: 0]$. 

\item For every line $i,j,k$ in $\bA^2_{\bF_3}$, we get a singular point by setting $z_\ell = 0$ for $\ell \notin \{i,j,k\}$ and setting $z_i = 1$, $z_j = \zeta$ and $z_k = \zeta'$ where $\zeta,\zeta'$ are arbitrary cube roots of unity. This gives $12 \cdot 9 = 108$ more singular points.

\item If all of the coordinates are nonzero, then after scaling one can see that they must be cube roots of unity. Write $z_i = \omega^{e_i}$ where $\omega$ is a fixed cube root of unity. Then we see that the condition that the first two rows of the matrix are collinear is that $z_1z_2z_3 = z_4z_5z_6 = z_7z_8z_9$, and we get $3$ more pairs of equations by comparing the first row with the remaining three rows. This gives us $8$ linear equations over $\bZ/3$ for the $e_i$, and in fact this system reduces to three equations:
\begin{align*}
e_1 + e_2 + e_3 = e_4 + e_5 + e_6, \quad 
e_1 + e_2 + e_3 = e_7 + e_8 + e_9, \quad
e_1 + e_4 + e_7 = e_2 + e_5 + e_8
\end{align*}
Hence we get $3^5 = 243$ solutions (with the requirement that $z_1 = 1$). 
\end{compactenum}

A direct check (we used {\tt Macaulay 2}) shows that none of these $360$ points lie on a smooth Abelian surface. Namely, given a point, we can solve for all possible $c_1, \dots, c_4$ for which it lies in the locus of the Jacobian ideal \eqref{eqn:jacobian}, and in all cases, $c_1, \dots, c_4$ lies on the discriminant locus \eqref{eqn:discw39}.

The Betti table of the ideal of $2 \times 2$ minors (which is radical) is
\small \begin{Verbatim}[samepage=true]
       0   1   2    3    4    5    6   7  8
total: 1 180 904 2180 3375 3265 1764 450 31
    0: 1   .   .    .    .    .    .   .  .
    1: .   .   .    .    .    .    .   .  .
    2: .   .   .    .    .    .    .   .  .
    3: . 180 684  920  405   45    .   .  .
    4: .   . 220 1260 2970 3220 1764 450 30
    5: .   .   .    .    .    .    .   .  .
    6: .   .   .    .    .    .    .   .  .
    7: .   .   .    .    .    .    .   .  1
\end{Verbatim}
\normalsize

Here is its Hilbert series
\[
\frac{1+8T+36T^2+120T^3+150T^4+36T^5+8T^6+T^7}{1-T}.
\]

\begin{remark}
Each plane in the $3 \times 3$ minors contains exactly $12$ points in the $2 \times 2$ minors. Also, each point is contained in exactly $4$ planes:

The points of type 1 are contained in $4$ planes of type 1.

The points of type 2 are contained in $1$ plane of type 1 and $3$ planes of type 2.

The points of type 3 are contained in $4$ planes of type 2. 
\end{remark}

\noindent {\bf Authors' addresses:}

\medskip

\small \noindent Laurent Gruson, 
Universit\'e de Versailles Saint-Quentin-en-Yvelines, 
Versailles, France \\
{\tt laurent.gruson@math.uvsq.fr}

~

\small \noindent Steven V Sam, 
University of California, Berkeley, CA, USA\\
{\tt svs@math.berkeley.edu}, \url{http://math.berkeley.edu/~svs/}

\end{document}